\documentclass[10pt,3p, times,fullpage]{article}

\usepackage[mathscr]{eucal}
\usepackage{amsmath,amsxtra,amssymb,latexsym,amscd,amsthm,amsfonts,amstext,makeidx,upgreek,dsfont,eufrak,xcolor}
\usepackage[top= 1.9 cm,bottom=1.9 cm,left=2.5cm,right=2cm]{geometry}
\usepackage{indentfirst}
\usepackage{enumerate}

\usepackage{hyperref,url} 
\usepackage{graphicx}
\usepackage{subfloat,subfigure}
\usepackage{makeidx}

\usepackage{multirow}
\usepackage{hhline}
\usepackage{epstopdf}

\newcommand{\mb}[1] {\mathbb{#1}}
\newcommand{\mc}[1] {\mathcal{#1}}
\newcommand{\mf}[1] {\mathfrak{#1}}
\newcommand{\ms}[1] {\mathscr{#1}}
\newcommand{\mbf}[1] {\mathbf{#1}}

\newcommand{\nn} {\nonumber}

\newcommand{\al} {\alpha}

\newcommand{\be} {\beta}

\usepackage[utf8]{inputenc} 
\def\d{\mathrm{d}} 

\theoremstyle{definition}
\newtheorem{example}{Example}[section]
\theoremstyle{plain}
\newtheorem{lemma}{Lemma}[section]
\newtheorem{theorem}{Theorem}[section]

\newtheorem{remark}{Remark}[section]

\title{\bf Final value problem for nonlinear time fractional reaction-diffusion equation with discrete data}
\author{Nguyen Huy Tuan$^{1}$, Dumitru Baleanu$^{2,3,4}$, Tran Ngoc Thach$^{5}$, \\
		 Donal O’Regan$^{6}$ and Nguyen Huu Can$^{7,}$\footnote{Correspondence: nguyenhuucan@tdtu.edu.vn (Nguyen Huu Can). E-mail of other authors: thnguyen2683@gmail.com (Nguyen Huy Tuan), dumitru@cankaya.edu.tr (Dumitru Baleanu), ngocthachtnt@gmail.com  (Tran Ngoc Thach), donal.oregan@nuigalway.ie (Donal O’Regan).}\\\\
	\small $^{1}$ Institute of Research and Development, Duy Tan University, Da Nang 550000, Vietnam\\
\small $^2$ Department of Mathematics, Cankaya University, Ankara, Turkey\\
\small $^3$ Department of Medical Research, China Medical University Hospital, China Medical University, Taichung, Taiwan\\
\small $^4$ Institute of Space Sciences, Magurele–Bucharest, Romania\\
\small $^{5}$ Department of Mathematics and Computer Science, University of Science, Ho Chi Minh City, Vietnam,\\
\small Vietnam National University, Ho Chi Minh City, Vietnam\\
\small $^6$ School of Mathematics, Statistics and Applied
Mathematics, \\
	\small 	National University of Ireland, Galway, Ireland\\
	\small $^7$ Applied Analysis Research Group, Faculty of Mathematics and Statistics, \\
	\small
	Ton Duc Thang University, Ho Chi Minh City, Vietnam
}

\begin{document}
\maketitle
\begin{abstract}
In this paper, we study the problem of finding the solution of a multi-dimensional time fractional reaction-diffusion equation with nonlinear source from the final value data.  We prove that the present problem is not well-posed. Then regularized problems are constructed  using the truncated expansion method (in the case of two-dimensional) and the quasi-boundary value method (in the case of multi-dimensional). Finally, convergence rates of the regularized solutions are given and investigated numerically.

\end{abstract}

\section{Introduction}
{\color{red} 
In recent times, fractional partial differential equations have received great attention both in analysis and application, which are used in modeling several phenomena in different areas of science such as physics, biology, chemistry, engineering and control theory, see \cite{Deb,Gor,Gup,Pod,Sam,jin, New1, New2, New3, New4}, so the fractional computation is increasingly attracted to mathematicians, where some circumstances integer-order partial differential equations cannot simulate. Accordingly, many definitions of fractional derivative are given \cite{Kil,Pod,ZouY}.
}
In this paper, we study a problem for the time fractional reaction-diffusion equation with nonlinear source
\begin{align} \label{pr} \begin{cases}   \dfrac{\partial}{\partial t} u(t,x) - \dfrac{\partial^{1-\al}}{\partial t}  \Delta  u(t,x) = f (t,x,u(t,x)), &(t,x)  \in (0,\mc{T})  \times  \Omega, \al \in (0,1) \\
u(t,x) = 0, &(t,x)  \in (0,\mc{T})  \times \partial \Omega, \\
u(\mc{T},x) = \varphi(x), &x \in \Omega,
\end{cases}
\end{align} 
where $\Omega = (0, \pi)^d$ is a subset of $\mb{R}^d$, $x=(x_1,x_2,\dots,x_d)$ is a $d$-dimensional variable. The source function $f$ is given and $\varphi$ is called the final value status. The notation  $\frac{\partial^{1-\al}}{\partial t}$ denotes the R-L fractional derivative
\begin{align*}
\dfrac{\partial^{1-\al}}{\partial t}    
v(t) = \dfrac{1}{\Gamma(\al)} \dfrac{d}{dt} \int\limits_{ 0 }^{ t } (t-s)^{\al-1} v(s) \d s, \quad t>0,
\end{align*} 
where $\Gamma(\cdot)$ is the Gamma function. It should be noted that if $\al =1$ then $\frac{\partial^{1-\al}}{\partial t} \Delta u$ becomes $\Delta u$ and the equation 
\begin{align}   \label{fde} 
\dfrac{\partial}{\partial t} u(t,x) - \dfrac{\partial^{1-\al}}{\partial t}  \Delta  u(t,x) = f (t,x,u(t,x)),
\end{align}  
reduces to the typical heat equation (note  \eqref{fde} comes from \cite{Sch}).  Schneider and Wyss \cite{Sch}  showed that the description of
diffusion in special types of porous media is an application of \eqref{fde} and that the fractional parameter $\alpha \in (0,1)$ can represent the “gray”
noise instead of the white noise in the case of $\alpha=1$.

Equation \eqref{fde} with the initial condition $u(0,x)=\psi(x)$  is known as the direct problem. For more details of this equation we refer  the reader  to \cite{And,Sch,Sek} and the references therein. Numerical solutions of the alternative representation of such direct problem have been studied in \cite{Chen,Cuesta,Mus,jin2}.  In contrast, the problem of recovering the function $u$ at previous time $t \in [0,\mc{T})$ as in \eqref{pr} is called the backward problem. This kind of equation arises in practical situations  in which the initial density of the diffusing substance is not available and we can only measure the density  at positive time. { We  mention the applications of backward in time diffusion equations  in  the work of A. S. Carraso \cite{Cas,Cas1}. 
Two  major current applications of the backward problem  are hydrologic inversion and image deblurring. 
Hydrologic inversion seeks to identify sources of groundwater pollution by backtracking contaminant
plumes (see \cite{At,At1}) and this involves solving the  diffusion backward in time,
given the contaminant spatial distribution  at the current time $\mathcal{T}$.   In image analysis, an effective setting for studying  2-D backward  diffusion   lies in the field of imaging
rehabilitation.     One can create imaginary fuzzy image data, using a certain sharp image as the initial value in the nonlinear diffusion equation 
studied and select the corresponding solution in a positive number
time T so successful backward continuation from $t = T $ to $t = 0,$ would restore the original sharp image. 
} Until now, there are some interesting papers on inverse problem of fractional diffsuion. We can list some well-known results, for example, J. Jia et al \cite{Jia},  J. Liu et al \cite{Liu}, some papers of M. Yamamoto and his group see \cite{Ya5,Ya6,Ya8,Ya9,Ya10},   B. Kaltenbacher et al \cite{Bar,Bar1,Bar8,Bar9},   W. Rundell et al \cite{Run,Run1},    J. Janno see \cite{Ja1,Ja2}, etc.
 However, to the best of our knowledge, there is no result concerning  the backward problem for \eqref{fde} with random noise.

Motivated by the  above, in this paper, we study  problem \eqref{pr} and aim to provide an approximate solution. In reality, it is impossible to get the exact final data $\varphi$ and we only have the noisy physical measurement $\widetilde{\varphi}$. A difficult point of the backward problem is  a small noise between $\widetilde{\varphi}$ and $\varphi$ can generate a very large error in the solution $u$. In other words, the solution does not depend continuously on the final value status (which makes \eqref{pr}  not well-posed). Therefore, we must provide some suitable methods to find an approximation for $u$. When  the final value status is measured on the whole space $\Omega$ many good methods can be applied to establish the approximate regularized solution such as the Tikhonov, the quasi-boundary value (QBV), the quasi-reversibility (QR) and the truncated expansion method (see \cite{Che, Den, Nam, Qia}). Here we consider a different situation in which only a finite number of data (instead of  data on the whole space) is available.  Precisely, we assume that the data $\varphi$ is measured at $n_1 \times n_2 \times \dots \times n_d$ grid points $x_{\mbf{k}} = x_{k_1,k_2,\dots,k_d} \in \Omega$,$d \ge 2$,  $\mbf{k}=(k_1,k_2,\dots,k_d) \in \mb{N}^d$,  as follows
\begin{align*}    
x_{\mbf{k}} = (\mc{X}_{k_1},\mc{X}_{k_2}, \dots, \mc{X}_{k_d}) = \left(  \dfrac{2 k_1 -1}{2 n_1} \pi,  \dfrac{2 k_2 -1}{2 n_2} \pi , \dots, \dfrac{2 k_d -1}{2 n_d} \pi \right),
\end{align*} 
where $k_i = 1,2,\dots,n_i$, $i = 1,2,\dots,d$. Furthermore, the value of $\varphi$ at each  point $x_{\mbf{k} }$ is contaminated by the observation $\Phi^{obs}_{\mbf{k} }$ 
\begin{align*}    
\varphi(x_{\mbf{k}})=\varphi\left( \mc{X}_{k_1},\mc{X}_{k_2},\dots,\mc{X}_{k_d} \right) ~ \approx ~  \Phi^{obs}_{k_1,k_2,\dots,k_d} = \Phi^{obs}_{\mbf{k}}.
\end{align*}  
The relationship between the two kinds of data is described by the random model
\begin{align}    \label{model}    \Phi^{obs}_{\mbf{k}} = \varphi(x_{\mbf{k}}) + \varepsilon_{\mbf{k}} \ms{W}_{\mbf{k}},
\end{align}   
where $\ms{W}_{\mbf{k}}=\ms{W}_{k_1,k_2,\dots,k_d}$ are mutually independent random variables, $\ms{W}_{\mbf{k}}\sim \mc{N}(0,1)$ and $\varepsilon_{\mbf{k}}=\varepsilon_{k_1,k_2,\dots,k_d}$ are positive constants bounded by a positive constant $\varepsilon_{\max}$. Some inverse problems when $d=1$ were studied in  \cite{Bis,Nan, Yang}.

 Our main contributions in this paper is as follows:
\begin{itemize}
	\item For the two dimensional case, i.e, $d=2$, we apply the Fourier truncation method  introduced in \cite{Minh} to give a regularized problem.  {
	The model in \cite{Minh} is linear. Our problem is nonlinear and we  use the Banach fixed point theorem to show  the existence of the regularized solution in the space $ \mb{X}_\mc{T}$ (note this space does not appear in \cite{Minh}).  Some new estimates of Mittag-Leffler type are  used. 
}

	\item For the multidimensional case with $d >2$, we apply the quasi-boundary value method (QBV).  We
	 emphasize that our random model here is a multidimensional case which is a generalization of the results in \cite{Bis,Nan, Yang}. Our method in this case  is new and very different from the methods in \cite{Minh}. First, we approximate H and $\varphi$ by the approximating functions $\widehat{\varphi}^{\gamma_{\mbf{n}} }$ 
	 defined in Theorem 5.1. Then, we use the approximation data to establish a regularized solution using the QBV method.  {Moroever, we also give a new filter method which contains some results on the truncation and quasi-boundary value method (this filter is a new contribution). }  In particular  in our error estimates we show that the
	norm of the difference between the regularized solution  and the solution
	of the problem \eqref{pr} tends to zero when $\sqrt{n_1^2+...+n_d^2} \to +\infty$. 
\end{itemize}

The structure of this paper is as follows. We first give some preliminaries which are needed for this paper in Section \ref{ss2}. In Section \ref{ss3}, we establish an integral formulation for the solution of  problem \eqref{pr}. In Section \ref{sec3}, we prove that the present problem is not well-posed and then we construct an approximate regularized solution for the $2$-dimensional problem using the Fourier truncated method. The convergence result is also given there. In Section \ref{s5}, the multi-dimensional problem is considered and regularized  using the quasi-boundary value method. We  estimate the error between the approximation and the sought solution in two different spaces. Finally, we
provide some numerical results to illustrate the convergence rates.
\section{Preliminaries} \label{ss2}
Before going to the main parts, we present some concepts:
\begin{itemize}
	\item For $\mbf{j}=(j_1,j_2,\dots,j_d) \in \mb{N}^d$, we denote
	$   
	|\mbf{j}|=\sqrt{\sum_{ i=1 }^{ d } j_i^2 }.
	$
	It is well-known that the following problem
	\begin{align*}
	\begin{cases}         
	-\Delta  {\xi_{\mbf{j}}(x)} = \lambda_{\mbf{j}} \xi_{\mbf{j}}(x), & x \in \Omega, \\
	\xi_{\mbf{j}}(x) = 0, & x \in \partial \Omega.
	\end{cases}     
	\end{align*}   
	admits eigenvalues $\left\{ \lambda_{\mbf{j}} \right\} $ and  eigenvectors $\left\{ \xi_{\mbf{j}} \right\} $ as follows
	\begin{align*}  
	\lambda_{\mbf{j}} = \lambda_{j_1,\dots,j_d}  =\sum_{ i=1 }^{ d } j_i^2 =|\mbf{j}|^2, \quad \xi_{\mbf{j}}(x)= \xi_{j_1,\dots,j_d}(x) = \left( \sqrt{\frac{2}{\pi}} \right) ^{d} \prod_{i=1}^{d} \sin (j_i x_i).
	\end{align*}  
	\item For $\al >0$ and  {$\be >0$}, the  function defined as follows
	\begin{align} \label{Mit}
	  E_{\al,\be} (z)= \sum_{ i=0 }^{ \infty } \dfrac{z^i}{\Gamma(\al i + \be)},  \quad z \in \mb{C}.
	\end{align} 
	is called the Mittag-Leffler function. Some properties of this function can be found in \cite{Pod}.
\item We introduce the subspace of $L^2(\Omega)$ 
\begin{align*}    
H^\theta (\Omega) = \left\{ g \in L^2(\Omega):  \sum_{ \mbf{j} \in \mb{N}^d }  \lambda_{\mbf{j} }  ^\theta \left\langle g,\xi_{\mbf{j} } \right\rangle ^2   < \infty \right\}, \quad  \theta >0,
\end{align*}  
with the norm
$
\left\| g \right\|_{H^\theta(\Omega)} = \left( \sum_{ \mbf{j} \in \mb{N}^d }  \lambda_{\mbf{j} }  ^\theta \left\langle g,\xi_{\mbf{j} } \right\rangle ^2 \right)^{1/2},
$
in which $\left\langle \cdot,\cdot \right\rangle $ denotes the inner product in $L^2(\Omega)$. 
\item For an arbitrary Banach space $\mbf{B}$, we set
\begin{align*}    
L^\infty \left( {0,\mc{T};\mbf{B}} \right) = \left\{ h: (0,\mc{T}) \to \mbf{B} \text{ measurable } \text{ s.t. } \text{esssup}_{t \in (0,\mc{T})} \left\| h(t,\cdot) \right\|_{\mbf{B}} < \infty \right\}. 
\end{align*}
\item    
We denote by $\mb{X}_\mc{T}$ (see \cite{Bae}), the space of all $L^2$-valued predictable processes $w$ such that
\begin{align*}  
\left\| w \right\|_{\mb{X}_{\mc{T}}} = \sup_{t \in [0,\mc{T}]} \sqrt{\mb{E} \left\| w(t,\cdot) \right\|^2_{L^2(\Omega)} } < \infty.
\end{align*}
For $\sigma>0$, we  denote by $\mb{S}_{\sigma,\mc{T}}$, the space of all $H^\sigma$-valued predictable processes $w$ such that
\begin{align*}  
\left\| w \right\|_{\mb{S}_{\sigma,\mc{T}}} = \sup_{t \in [0,\mc{T}]} \sqrt{\mb{E} \left\| w(t,\cdot) \right\|^2_{H^\sigma(\Omega)} } < \infty.
\end{align*}
\end{itemize}
\section{The solution of the backward problem} \label{ss3}
Let
$u(t,x)= \sum_{\mbf{j} \in \mb{N}^d}  u_\mbf{j}(t)  \xi_\mbf{j}(x)$ be the Fourier   series of $u$ in  $L^2(\Omega)$, where $ u_{\mbf{j}}(t)  :=\left\langle u(t,\cdot), \xi_{\mbf{j}}  \right\rangle $ are called the Fourier coefficients of $u$. Similarly, we denote $h_{\mbf{j}} := \left\langle h, \xi_{\mbf{j}}  \right\rangle  $ and $f_{\mbf{j}}(u)(t):=\left\langle f(t,\cdot,u(t,\cdot)),\xi_{\mbf{j}} \right\rangle $.  
Next, we find a representation for the solution $u$ of  problem \eqref{pr}. By taking
the inner product on both sides of the first equation in \eqref{pr} with  {$\xi_{\mbf{j}}$}, one has
\begin{align*}    
\frac{\partial}{\partial t} u_{\mbf{j}}(t)  {+} \lambda_{\mbf{j}} \frac{\partial ^{1-\al}}{\partial t} u_{\mbf{j}}(t) = f_{\mbf{j}}(u)(t).
\end{align*} 
Using the Laplace transform and  $u_{\mbf{j}}(\mc{T})=\varphi_{\mbf{j}}$, one obtains
\begin{align*}    
\widehat  {u_{\mbf{j}}}(s) = \dfrac{s^{\al-1}}{s^\al + \lambda_{\mbf{j}}} u_{\mbf{j}}(0) + \dfrac{s^{\al-1}}{s^\al + \lambda_{\mbf{j}}} \widehat  { f_{\mbf{j}}(u) } (s),
\end{align*} 
where the notation $\widehat {g}$ is the Laplace transform of $g$. In order to solve the equation above, we need some following properties of the Mittag-Leffler function (see \cite{Pod})
\begin{align*}    
&\dfrac{\partial}{\partial t}   
\left( E_{\al,1}(-\lambda t^\al) \right) = - \lambda t^{\al-1} E_{\al,\al} (-\lambda t^\al), \quad \dfrac{\partial}{\partial t} \left( t^{\al -1} E_{\al,\al}(-\lambda t^\al) \right) = t^{\al-2} E_{\al,\al-1} (-\lambda t^\al),
\end{align*}  
and
\begin{align*}    
\displaystyle \int\limits_{ 0 }^{ \infty } e^{-s t} E_{\al,1} (-\lambda t^\al) \d t = \dfrac{s^{\al-1}}{s^\al +\lambda}, \quad \text{ for Re} (s) > \lambda^{1/\al},
\end{align*}
in which $\lambda$ is a positive constant. In this way, one gets
\begin{align}  \label{ummm} 
u_{\mbf{j}} (t) = E_{\al,1} \big(- \lambda_{\mbf{j}} t^\al \big) u_{\mbf{j}}(0) + \int\limits_{ 0 }^{ t } E_{\al,1} \big( -\lambda_{\mbf{j}} (t-s)^\al \big)  f_{\mbf{j}}(u)(s) \d s.
\end{align}
It follows that
\begin{align*}   
u_{\mbf{j}} (\mc{T}) = E_{\al,1} \big(- \lambda_{\mbf{j}} \mc{T}^\al \big) u_{\mbf{j}}(0) + \int\limits_{ 0 }^{ \mc{T} } E_{\al,1} \big( -\lambda_{\mbf{j}} (\mc{T}-s)^\al \big)  f_{\mbf{j}}(u)(s) \d s.
\end{align*}  
Using the latter equation to determine  $u_{\mbf{j}}(0)$ based on $u_{\mbf{j}}(\mc{T})$ and then substituting this quantity into \eqref{ummm}, the Fourier coefficients are obtained
\begin{align}  \label{ucoe}
u_{\mbf{j}}(t) 
&=  \dfrac{E_{\al,1} \big(- \lambda_{\mbf{j}} t^\al \big)}{E_{\al,1} \big(- \lambda_{\mbf{j}} \mc{T}^\al \big)} \varphi_{\mbf{j}} 
+ \int\limits_{ 0 }^{ t } E_{\al,1} \big( -\lambda_{\mbf{j}} (t-s)^\al \big) f_{\mbf{j}}(u)(s) \d s \nn
\\
&-  \int\limits_{ 0 }^{ \mc{T} } \dfrac{E_{\al,1} \big(- \lambda_{\mbf{j}} t^\al \big) E_{\al,1} \big( -\lambda_{\mbf{j}}(\mc{T}-s)^\al \big)}{E_{\al,1} \big(- \lambda_{\mbf{j}} \mc{T}^\al \big)} f_{\mbf{j}}(u)(s) \d s.
\end{align}
For convenience, we define the operators
\begin{align}   
\mc{A}(t) {g} &:= \sum_{ \mbf{j}  \in \mbf{N}^d } \Bigg( \dfrac{E_{\al,1} \big(- \lambda_{\mbf{j} } t^\al \big)}{E_{\al,1} \big(- \lambda_{\mbf{j}} \mathcal{T}^\al \big)} \left\langle {g}, \xi_{\mbf{j}} \right\rangle  \Bigg) \xi_{\mbf{j}}, \label{defa} \\
\mc{B}(t) {g} &:= \sum_{ \mbf{j}  \in \mbf{N}^d } \Big(  E_{\al,1} \big( -\lambda_{\mbf{j} } t^\al \big) \left\langle {g}, \xi_{\mbf{j} } \right\rangle \Big) \xi_{\mbf{j} } \label{defb} ,
\end{align}
and $\mc{D}(t,s){g}:=\mc{A}(t) \mc{B}(s){g}$, for ${g} \in L^2(\Omega)$, $t,s \in [0,\mc{T}]$. Now, we conclude that $u$ satisfies the equation
\begin{align}  \label{opereq}       u(t,x) = \mc{A}(t) \varphi(x) + \int\limits_{ 0 }^{ t } \mc{B}(t-s) f(s,x,u(s,x)) \d s -  \int\limits_{ 0 }^{ \mc{T} } \mc{D}(t,\mc{T}-s)  f(s,x,u(s,x)) \d s.
\end{align}

The following lemma (see \cite{Pod}) is useful for estimating the Mittag-Leffler function:
\begin{lemma}  \label{lm11}
	Let $0 < \al <1$ and the function $E_{\al,\al}$ defined in \eqref{Mit}. 
	For a real number $z >0$, we have
	\begin{align*}    
	\dfrac{\ms{M}_{1}}{1+z} \le E_{\al,\al}(-z) \le \dfrac{\ms{M}_{2}}{1+z},
	\end{align*}  
	in which $\ms{M}_1=\ms{M}_1(\al)$ and $\ms{M}_2=\ms{M}_2(\al)$ are positive constants depending only on $\al$.
\end{lemma}
Using the latter lemma, we give some properties for the operators $\mc{B}, \mc{D}$ as follows:
\begin{lemma} \label{lmop}
	Let $g$ be a function in $L^2(\Omega)$. For $t \in [0,\mc{T}]$,  $s \in (0,\mc{T}]$, we have
	\begin{align*}    
	\left\| \mc{B}(t) g \right\|_{L^2(\Omega)}   \le \ms{M}_2 \left\| g \right\|_{L^2(\Omega)} ,\left\| \mc{D}(t,s) g  \right\|_{L^2(\Omega)}  \le  \dfrac{\ms{M}^2_2 \mc{T}^\al}{\ms{M}_1 s^\al} \left\| g  \right\|_{L^2(\Omega)} .
	\end{align*}     
\end{lemma}
\begin{proof}
	Using Lemma \ref{lm11}, we have
	\begin{align*}    
	\left\| \mc{B}(t) g \right\|_{L^2(\Omega)}^2  = \sum_{ \mbf{j} \in \mb{N}^d } E^2_{\al,1} \big(- \lambda_{\mbf{j}} t^\al \big)  \left\langle g, \xi_{\mbf{j}} \right\rangle^2 \le  \ms{M}_2^2  \left\| g \right\|_{L^2(\Omega)}^2 . 
	\end{align*} 
	For the second term, we  have
	\begin{align*}    
	\left\| \mc{D}(t,s) g  \right\|_{L^2(\Omega)}^2  &= \sum_{ \mbf{j} \in \mb{N}^d } \Bigg( \dfrac{E_{\al,1} \big(- \lambda_{\mbf{j}} t^\al \big) E_{\al,1} \big( -\lambda_{\mbf{j}}s^\al \big)}{E_{\al,1} \big(- \lambda_{\mbf{j}} \mc{T}^\al \big)}\Bigg) ^2 \left\langle g, \xi_{\mbf{j}} \right\rangle^2 \\
	&\le \sum_{ \mbf{j} \in \mb{N}^d } \left( \dfrac{\ms{M}^2_2 (1+\lambda_{\mbf{j}} \mc{T}^\al)}{\ms{M}_1 (1+\lambda_{\mbf{j}} s^\al)} \right)^2   \left\langle g, \xi_{\mbf{j}} \right\rangle^2  \le \left( \dfrac{\ms{M}^2_2 \mc{T}^\al}{\ms{M}_1 s^\al} \right)^2   \left\| g  \right\|_{L^2(\Omega)}^2 . 
	\end{align*}
	This completes the proof.   
\end{proof}
\section{Backward problem in the case of two-dimensional}
 \label{sec3}
In this section, we study problem \eqref{pr} when $d=2$. Here, we recall that $\Omega = (0, \pi) \times (0, \pi)$, $x=(x_1,x_2)$ is a $2$-dimensional variable. The value of $\varphi$ at $n_1 \times n_2$  grid points $$x_{k_1,k_2} = (\mc{X}_{k_1},\mc{X}_{k_2}) =\left(  \dfrac{2 k_1 -1}{2 n_1} \pi,  \dfrac{2 k_2 -1}{2 n_2} \pi \right), \quad k_1 = 1,2,\dots, n_1,~k_2 = 1,2,\dots, n_2,$$
are contaminated by the observed data $\Phi^{obs}_{k_1,k_2}$ as in the model
\begin{align}        \Phi^{obs}_{k_1,k_2} = \varphi(x_{k_1,k_2}) + \varepsilon_{k_1,k_2} \ms{W}_{k_1,k_2},
\end{align}    
where $\ms{W}_{k_1,k_2}$ are mutually independent random variables, $\ms{W}_{k_1,k_2} \sim \mc{N}(0,1)$ and $0 < \varepsilon_{k_1,k_2} \le \varepsilon_{\max}$. 

Our main aim is to show that the  problem is not well-posed and then to construct an approximate regularized solution using the Fourier truncation method.  
\subsection{Estimator for the final value data $\varphi$} \label{ss31}
It should be noted that the final data $\varphi$ on the whole space  $\Omega$ is not available. Therefore, we now establish an approximate function for $\varphi$ based on the observations and then estimate the error between them.
\begin{lemma} \label{appp}
Let $N_1=N_1(n_1)$, $N_2=N_2(n_2)$ be natural numbers less than $n_1,n_2$ respectively. Assume that there exists a constant $\theta >2$ such that $\varphi \in H^{\theta}(\Omega)$. Approximate $\varphi$ by the following function
\begin{align} \label{phitilde2d}   
\widetilde \varphi^{N_1,N_2}:=  \sum_{ j_1=1 }^{ N_1 } \sum_{ j_2=1 }^{ N_2 } \Bigg( \frac{\pi^2}{n_1n_2} \sum_{ k_1=1 }^{ n_1 } \sum_{ k_2=1 }^{ n_2}  \Phi_{k_1,k_2}^{obs}  \xi_{j_1,j_2}(x_{k_1,k_2}) \Bigg) \xi_{j_1,j_2}. 
\end{align} 
Then, we have
\begin{align*}    
\mb{E} \left\| \widetilde \varphi ^{N_1,N_2} - \varphi \right\|^2_{L^2(\Omega)}  \le \dfrac{\left( 2 \pi^2 \varepsilon^2_{\max} +  2 \mc{C}_0^2  \left\| \varphi \right\|^2_{H^\theta (\Omega)} \right) N_1 N_2 }{n_1 n_2} + 2 \left[ (N_1+1)^{-2 \theta} + (N_2+1)^{-2 \theta} \right] \left\| \varphi \right\|^2_{H^\theta(\Omega)},
\end{align*} 
where $\mc{C}_0$ is a positive constant (see \eqref{C0}).
\end{lemma}
\begin{remark}
If we choose $N_1=N_1(n_1),N_2=N_2(n_2)$ satisfying
\begin{align*}
\lim_{n_1 \to \infty} N_1= \lim_{n_2 \to \infty} N_2 = \infty, \quad \lim_{n_1,n_2 \to \infty} \frac{N_1N_2}{n_1n_2}=0,
\end{align*} 
then $\mb{E} \left\| \widetilde \varphi ^{N_1,N_2} - \varphi \right\| ^2_{L^2(\Omega)}  $ tends to zero as $n_1,n_2$ tend to infinity. 
\end{remark}
The following Lemma is useful for proving Lemma \ref{appp}:
\begin{lemma} \label{lm22} Let $m_1,m_2$ be natural numbers less than $n_1,n_2$ respectively. Set
\begin{align*}    
\mf{C}_{j_1,j_2,m_1,m_2} := \sum_{ k_1=1 }^{ n_1 } \sum_{ k_2=1 }^{ n_2 }  \xi_{j_1,j_2}(x_{k_1,k_2}) \xi_{m_1,m_2}(x_{k_1,k_2}). 
\end{align*} 
Then, we have  
	\begin{align*}    
	\mf{C}_{j_1,j_2,m_1,m_2} = \begin{cases}         
	\dfrac{(-1)^{l_1+l_2}}{\pi^2} n_1 n_2, &\text{if } (j_1,m_1) = \pm (j_2,m_2) + (2l_1n_1,2l_2n_2), \\
	-\dfrac{(-1)^{l_1+l_2}}{\pi^2} n_1 n_2, &\text{if } (j_1,m_1) = \pm (-j_2,m_2) + (2l_1n_1,2l_2n_2), \\
	0, &\text{otherwise}.
	\end{cases} 
	\end{align*} 
\end{lemma} 
\begin{proof} [Proof of Lemma \ref{lm22}]	
	This Lemma can be prove using the method  in \cite{Eub}, page 145.
\end{proof}
\begin{proof} [Proof of Lemma \ref{appp}]
Recall that $\varphi_{j_1,j_2}  = \left\langle \varphi, \xi_{j_1,j_2} \right\rangle$, for $j_1,j_2 \in \mb{Z}^+$. Using \eqref{phitilde2d} and the fact that
$$    
\varphi(x) = \sum_{ j_1=1 }^{ \infty } \sum_{ j_2=1 }^{ \infty } \varphi_{j_1,j_2} \xi_{j_1,j_2}(x),
$$ 
we have
\begin{align} \label{4e}   
\left\| \widetilde \varphi ^{N_1,N_2} - \varphi \right\| ^2_{L^2(\Omega)} &=  \sum_{ j_1=1 }^{ N_1 } \sum_{ j_2=1 }^{ N_2 } \Bigg( \frac{\pi^2}{n_1n_2} \sum_{ k_1=1 }^{ n_1 } \sum_{ k_2=1 }^{ n_2}  \Phi_{k_1,k_2}^{obs}  \xi_{j_1,j_2}(x_{k_1,k_2}) - \varphi_{j_1,j_2} \Bigg)^2 \nn \\
&+ \sum_{ j_1=1 }^{ N_1 } \sum_{ j_2=N_2+1 }^{ \infty } \varphi_{j_1,j_2}^2 + \sum_{ j_1=N_1+1 }^{ \infty } \sum_{ j_2=1 }^{ N_2 } \varphi_{j_1,j_2}^2 + \sum_{ j_1=N_1+1 }^{ \infty } \sum_{ j_2=N_2+1 }^{ \infty } \varphi_{j_1,j_2}^2 \nn \\
&=: \mc{E}_1 + \mc{E}_{2,1} + \mc{E}_{2,2} + \mc{E}_{2,3} . 
\end{align}
\textbf{Part A} (Estimating $\mc{E}_1$). Construct estimators for the coefficients $\varphi_{j_1,j_2}$, for $j_1 \le N_1, j_2 \le N_2$, as follows 
\begin{align*}    
\varphi_{j_1,j_2} \approx \dfrac{\pi^2}{n_1 n_2} \sum_{ k_1=1 }^{ n_1 } \sum_{ k_2=1 }^{ n_2 } \varphi(x_{k_1,k_2}) \xi_{j_1,j_2}(x_{k_1,k_2}),
\end{align*}  
and set
\begin{align} \label{upsilon}     \Upsilon_{j_1,j_2}:= \dfrac{\pi^2}{n_1 n_2} \sum_{ k_1=1 }^{ n_1 } \sum_{ k_2=1 }^{ n_2 } \varphi(x_{k_1,k_2}) \xi_{j_1,j_2}(x_{k_1,k_2}) - \varphi_{j_1,j_2}.
\end{align} 
Then, we get
\begin{align*}    
\mc{E}_1
&= \sum_{ j_1=1 }^{ N_1 } \sum_{ j_2=1 }^{ N_2 } \Bigg( \frac{\pi^2}{n_1n_2} \sum_{ k_1=1 }^{ n_1 } \sum_{ k_2=1 }^{ n_2}  \varepsilon_{k_1,k_2} \ms{W}_{k_1,k_2} \xi_{j_1,j_2}(x_{k_1,k_2}) + \Upsilon_{j_1,j_2} \Bigg)^2 \\
& \le 2 \sum_{ j_1=1 }^{ N_1 } \sum_{ j_2=1 }^{ N_2 } \Bigg( \frac{\pi^2}{n_1n_2} \sum_{ k_1=1 }^{ n_1 } \sum_{ k_2=1 }^{ n_2}  \varepsilon_{k_1,k_2} \ms{W}_{k_1,k_2} \xi_{j_1,j_2}(x_{k_1,k_2}) \Bigg)^2 + 2 \sum_{ j_1=1 }^{ N_1 } \sum_{ j_2=1 }^{ N_2 }  \Upsilon_{j_1,j_2} ^2. \\
&=: \mc{E}_{1,1}+ \mc{E}_{1,2}.
\end{align*} 
The first quantity can be estimated as follows
\begin{align*}    
\mb{E}\mc{E}_{1,1} \le  \frac{2 \pi^4}{n_1^2n_2^2} \varepsilon^2_{\max}  \sum_{ j_1=1 }^{ N_1 } \sum_{ j_2=1 }^{ N_2 } \mb{E} \Bigg( \sum_{ k_1=1 }^{ n_1 } \sum_{ k_2=1 }^{ n_2}   \ms{W}_{k_1,k_2} \xi_{j_1,j_2}(x_{k_1,k_2}) \Bigg)^2.
\end{align*}   
Using the properties $\mb{E} \left( \ms{W}_{k_1,k_2} \ms{W}_{l_1,l_2} \right) = \delta_{k_1 l_1} \delta_{k_2 l_2} $ and $\sum_{ k_1=1 }^{ n_1 } \sum_{ k_2=1 }^{ n_2} \xi_{j_1,j_2}^2(x_{k_1,k_2}) = \mf{C}_{j_1,j_2,j_1,j_2} = \pi^{-2}n_1n_2$ (see Lemma \ref{lm22}), we obtain
\begin{align*}   
\mb{E} \mc{E}_{1,1} &\le  \frac{2 \pi^4}{n_1^2n_2^2} \varepsilon^2_{\max}  \sum_{ j_1=1 }^{ N_1 } \sum_{ j_2=1 }^{ N_2 } \Bigg( \sum_{ k_1=1 }^{ n_1 } \sum_{ k_2=1 }^{ n_2}  \mb{E} \ms{W}^2_{k_1,k_2}  \xi^2_{j_1,j_2}(x_{k_1,k_2}) \Bigg)  =  \frac{2 \pi^2   N_1 N_2}{n_1n_2} \varepsilon^2_{\max}.
\end{align*}
In order to estimate the term $\mb{E} \mc{E}_{1,2}$, we need to undergo some steps as follows:

\noindent \textbf{Step 1} (Finding an explicit form for the error $\Upsilon_{j_1,j_2}$). In this step, we will prove that
\begin{align} \label{upid}   
\Upsilon_{j_1,j_2} =  \sum_{ l_2=1 }^{ \infty } (-1)^{l_2}  \varphi_{ j_1, 2l_2n_2 \pm j_2} + \sum_{ l_1=1 }^{ \infty } (-1)^{l_1}  \varphi_{2l_1n_1 \pm j_1, j_2} +\sum_{ l_1=1 }^{ \infty } \sum_{ l_2=1 }^{ \infty }(-1)^{l_1+l_2}  \varphi_{ j_1, 2l_2n_2 \pm j_2},
\end{align}
where we define
\begin{align}  \label{defp1}  
\varphi_{2l_1n_1 \pm j_1,j_2}:= \varphi_{2l_1n_1 + j_1,j_2} - \varphi_{2l_1n_1 - j_1,j_2}, \quad \varphi_{j_1,2l_2n_2 \pm j_2}:= \varphi_{j_1,2l_2n_2 + j_2} - \varphi_{j_1,2l_2n_2 - j_2},
\end{align}
and
\begin{align}  \label{defp2}  
\varphi_{2l_1n_1 \pm j_1,2l_2n_2 \pm j_2}:= \varphi_{2l_1n_1 + j_1,2l_2n_2 + j_2} &- \varphi_{2l_1n_1 + j_1,2l_2n_2 - j_2} \nn \\
&+ \varphi_{2l_1n_1 - j_1,2l_2n_2 - j_2} - \varphi_{2l_1n_1 - j_1,2l_2n_2 + j_2}.
\end{align}  
Indeed, since $\varphi(x_{k_1,k_2}) = \sum_{ m_1=1 }^{ \infty } \sum_{ m_2=1 }^{ \infty } \varphi_{m_1,m_2} \xi_{m_1,m_2}(x_{k_1,k_2}),$
we have
\begin{align*}    
\sum_{ k_1=1 }^{ n_1 } \sum_{ k_2=1 }^{ n_2 } \varphi(x_{k_1,k_2}) \xi_{j_1,j_2}(x_{k_1,k_2}) &=   \sum_{ m_1=1 }^{ \infty } \sum_{ m_2=1 }^{ \infty } \varphi_{m_1,m_2} \Bigg[ \sum_{ k_1=1 }^{ n_1 } \sum_{ k_2=1 }^{ n_2 } \xi_{m_1,m_2}(x_{k_1,k_2})  \xi_{j_1,j_2}(x_{k_1,k_2}) \Bigg] \\
&= \sum_{ m_1<n_1 } \sum_{ m_2<n_2 } \varphi_{m_1,m_2} \mf{C}_{j_1,j_2,m_1,m_2} + \sum_{ m_1<n_1 } \sum_{ m_2 \ge n_2 } \varphi_{m_1,m_2} \mf{C}_{j_1,j_2,m_1,m_2} \\
&+ \sum_{ m_1 \ge n_1 } \sum_{ m_2 < n_2 } \varphi_{m_1,m_2} \mf{C}_{j_1,j_2,m_1,m_2} + \sum_{ m_1 \ge n_1 } \sum_{ m_2 \ge n_2 } \varphi_{m_1,m_2} \mf{C}_{j_1,j_2,m_1,m_2}.
\end{align*}    
Applying Lemma \ref{lm22}, we get
\begin{align*}    
\dfrac{\pi^2}{n_1 n_2} \sum_{ k_1=1 }^{ n_1 } \sum_{ k_2=1 }^{ n_2 } \varphi(x_{k_1,k_2}) \xi_{j_1,j_2}(x_{k_1,k_2}) &=  \varphi_{j_1,j_2} +  \sum_{ l_2=1 }^{ \infty } (-1)^{l_2}  \varphi_{ j_1, 2l_2n_2 \pm j_2} \\
&+ \sum_{ l_1=1 }^{ \infty } (-1)^{l_1}  \varphi_{2l_1n_1 \pm j_1, j_2} +\sum_{ l_1=1 }^{ \infty } \sum_{ l_2=1 }^{ \infty }(-1)^{l_1+l_2}  \varphi_{ j_1, 2l_2n_2 \pm j_2},
\end{align*} 
which gives us formula \eqref{upid}.

\noindent \textbf{Step 2} (Estimating $\mc{E}_{1,2}$). From \eqref{upid}, we  see
\begin{align*}    
\left| \Upsilon_{j_1,j_2} \right| \le \sum_{ l_2=1 }^{ \infty } \left| \varphi_{ j_1, 2l_2n_2 \pm j_2} \right|   + \sum_{ l_1=1 }^{ \infty } \left| \varphi_{2l_1n_1 \pm j_1, j_2} \right|   +\sum_{ l_1=1 }^{ \infty } \sum_{ l_2=1 }^{ \infty } \left| \varphi_{ j_1, 2l_2n_2 \pm j_2}  \right|, \quad j_1 \le N_1, j_2 \le N_2.
\end{align*} 
Since $\varphi \in H^\theta (\Omega)$, we have $\lambda_{j_1,j_2}^{\frac{\theta}{2}} \left| \varphi_{j_1,j_2} \right|  \le \left\| \varphi \right\|_{H^\theta (\Omega)}$ for $j_1,j_2 \in \mb{Z}^+$. It follows that 
\begin{align*}    
\left| \varphi_{ j_1, 2l_2n_2 \pm j_2} \right| \le \dfrac{2}{\left( j_1^2 + (2l_2 n_2 - j_2)^2 \right) ^{\frac{\theta}{2}} } \left\| \varphi \right\|_{H^\theta (\Omega)} \le 2 l_2^{-\theta} n_2^{-\theta} \left\| \varphi \right\|_{H^\theta (\Omega)} . 
\end{align*} 
Similarly, one gets
\begin{align*}    
\left| \varphi_{2l_1n_1 \pm j_1, j_2} \right| \le 2 l_1^{-\theta} n_1^{-\theta} \left\| \varphi \right\|_{H^\theta (\Omega)}, \quad \left| \varphi_{ j_1, 2l_2n_2 \pm j_2}  \right| \le 4 \left( l_1^2 n_1^2 + l_2^2 n_2^2 \right)^{-\frac{\theta}{2}} \left\| \varphi \right\|_{H^\theta (\Omega)}. 
\end{align*} 
Hence
\begin{align*}    
\left| \Upsilon_{j_1,j_2} \right|  \le \left( 2 n_2^{-\theta} \sum_{ l_2=1 }^{ \infty } l_2^{-\theta} +  2 n_1^{-\theta} \sum_{ l_1=1 }^{ \infty } l_1^{-\theta} + 4(n_1^{-\theta} + n_2^{-\theta} ) \sum_{ l_1=1 }^{ \infty } \sum_{ l_2=1 }^{ \infty } l_1^{-\frac{\theta}{2}} l_2^{-\frac{\theta}{2}} \right) \left\| \varphi \right\|_{H^\theta (\Omega)}.
\end{align*}
Put
\begin{align}  \label{C0}  
\mc{C}_0 := 2\sum_{ l_1=1 }^{ \infty } l_1^{-\theta} + 4 \sum_{ l_1=1 }^{ \infty } \sum_{ l_2=1 }^{ \infty } l_1^{-\frac{\theta}{2}} l_2^{-\frac{\theta}{2}}. 
\end{align} 
Then, one has
\begin{align*}    
\mc{E}_{1,2} = 2 \sum_{ j_1=1 }^{ N_1 } \sum_{ j_1=2 }^{ N_2 }  \Upsilon_{j_1,j_2}^2 \le  \dfrac{ 2 \mc{C}_0^2 N_1 N_2 }{n_1 n_2} \left\| \varphi \right\|^2_{H^\theta (\Omega)},
\end{align*}
 for $n_1,n_2$ large enough. Now, we conclude that
\begin{align} \label{e1}   
\mb{E} \mc{E}_1 \le \mb{E} \mc{E}_{1,1} + \mc{E}_{1,2} \le \dfrac{\left( 2 \pi^2 \varepsilon^2_{\max} +  2 \mc{C}_0^2  \left\| \varphi \right\|^2_{H^\theta (\Omega)} \right) N_1 N_2 }{n_1 n_2}. 
\end{align}  
\textbf{Part B} (Estimating $\mc{E}_{2,1}$ to $\mc{E}_{2,3}$). From the definition of $\mc{E}_{2,1}$, one can see that
\begin{align*}    
\mc{E}_{2,1} = \sum_{ j_1=1 }^{ N_1 } \sum_{ j_2=N_2+1 }^{ \infty } \left( j_1^2+j_2^2 \right) ^{-\theta} \lambda_{j_1,j_2}^{\theta} \varphi_{j_1,j_2}^2 \le (N_2+1)^{-2 \theta}  \left\| \varphi \right\|^2_{H^\theta(\Omega)}. 
\end{align*}       
Similarly, one obtains
\begin{align*}    
\mc{E}_{2,2} \le (N_1+1)^{-2 \theta}  \left\| \varphi \right\|^2_{H^\theta(\Omega)}, \quad \mc{E}_{2,3} \le \left[  (N_1+1)^{-2 \theta}+ (N_2+1)^{-2 \theta} \right]  \left\| \varphi \right\|^2_{H^\theta(\Omega)}. 
\end{align*}  
This leads to
\begin{align} \label{3e}   
\mc{E}_{2,1} +  \mc{E}_{2,2} + \mc{E}_{2,3} \le 2 \left[ (N_1+1)^{-2 \theta} + (N_2+1)^{-2 \theta} \right] \left\| \varphi \right\|^2_{H^\theta(\Omega)}.
\end{align}  
Now, from \eqref{4e}, \eqref{e1} and \eqref{3e}, we deduce that
\begin{align*}    
\mb{E} \left\| \widetilde \varphi ^{N_1,N_2} - \varphi \right\|_{L^2(\Omega)} ^2 \le \dfrac{\left( 2 \pi^2 \varepsilon^2_{\max} +  2 \mc{C}_0^2  \left\| \varphi \right\|^2_{H^\theta (\Omega)} \right) N_1 N_2 }{n_1 n_2} + 2 \left[ (N_1+1)^{-2 \theta} + (N_2+1)^{-2 \theta} \right] \left\| \varphi \right\|^2_{H^\theta(\Omega)}.
\end{align*}  
This completes the proof.
\end{proof} 
\subsection{The ill-posedness of problem with discrete data} \label{subsec33} 
This subsection is aimed to demonstrate that the solution of the present problem is not stable, which follows that our problem is not well-posed. The solution of our problem is called stable if for any sequence $\varphi^{n_1,n_2}$,  {we have}
\begin{align} \label{Ep}   
\lim_{n_1,n_2 \to \infty} \mb{E} \left\| \varphi^{n_1,n_2} - \varphi \right\|^2_{L^2(\Omega)}  = 0,
\end{align}  
 {implies}
\begin{align} \label{Eu} 
\lim_{n_1,n_2 \to \infty} \left\| u^{n_1,n_2} - u  \right\|_{\mb{X}_{\mc{T}}} = 0,
\end{align}  
in which $u^{n_1,n_2}$ satisfies the system
\begin{align} \label{prip} \begin{cases}   \dfrac{\partial}{\partial t} u^{n_1,n_2}(t,x) - \dfrac{\partial^{1-\al}}{\partial t}  \Delta  u^{n_1,n_2}(t,x) = f (t,x,u^{n_1,n_2}(t,x)), &(t,x)  \in (0,\mc{T})  \times  \Omega, \\
u^{n_1,n_2}(t,x) = 0, &(t,x)  \in (0,\mc{T})  \times \partial \Omega, \\
u^{n_1,n_2}(\mc{T},x) = \varphi^{n_1,n_2}(x), &x \in \Omega.
\end{cases}
\end{align} 
Now, we give an example showing that there exists a sequence $\varphi^{n_1,n_2}$ such that \eqref{Ep} holds but \eqref{Eu} does not:
\begin{example}
Let $0< \al < \frac{1}{2}$,   $\varphi=0$, $f(t,x,u(t,x))= \mf{K} u(t,x)$ with
\begin{align*}    
\mf{K} = \mf{K} \left( {\al,\mc{T}} \right) = \dfrac{\ms{M}_1}{2 \ms{M}_2 \mc{T}} \sqrt{\dfrac{1-2\al }{(1-2\al)  +\ms{M}^2_2}},
\end{align*}
and the observed discrete data is
$    
\Phi^{obs}_{k_1,k_2} = \frac{1}{\sqrt[4]{n_1n_2}} \ms{W}_{k_1,k_2}.
$ Based on the idea in Subsection \ref{ss31}, we construct the sequence $\left\{ \varphi ^{n_1,n_2} \right\}$ as follows
\begin{align}  \label{phinn}  
 \varphi^{n_1,n_2} = \sum_{ j_1=1 }^{ n_1-1 } \sum_{ j_2=1 }^{ n_2-1 } \Bigg( \frac{\pi^2}{n_1n_2} \sum_{ k_1=1 }^{ n_1 } \sum_{ k_2=1 }^{ n_2}  \Phi^{obs}_{k_1,k_2}  \xi_{j_1,j_2}(x_{k_1,k_2}) \Bigg) \xi_{j_1,j_2}.
\end{align}
By a similar calculation as in Subsection \ref{ss31}, one can check that 
\begin{align} \label{einput} 
\mb{E} \left\| \varphi^{n_1,n_2} - \varphi \right\|_{L^2(\Omega)}^2   = \mb{E} \Bigg[~ \sum_{ j_1=1 }^{ n_1-1 } \sum_{ j_2=1 }^{ n_2-1 } \Bigg( \frac{\pi^2}{n_1n_2} \sum_{ k_1=1 }^{ n_1 } \sum_{ k_2=1 }^{ n_2}  \Phi^{obs}_{k_1,k_2}  \xi_{j_1,j_2}(x_{k_1,k_2}) \Bigg)^2 ~ \Bigg] = \dfrac{\pi^2 (n_1-1)(n_2-1)}{n_1^{3/2} n_2^{3/2}},
\end{align} 
which implies that $\mb{E} \left\| \varphi^{n_1,n_2} - \varphi \right\|_{L^2(\Omega)}^2  $ tends to zero as $n_1,n_2 \to \infty$.

Next, we show that $\left\| u^{n_1,n_2} -  u \right\|_{\mb{X}_\mc{T}}$ tends to infinity as $n_1,n_2 \to \infty$ . To do this, we first prove that $\mb{E} \left\| u(t,\cdot) \right\|^2_{L^2(\Omega)}  =0$, which implies that $u \equiv 0$. Indeed, from \eqref{opereq} and $\varphi=0$, one has
\begin{align} \label{uip}
u(t,x) = \int\limits_{ 0 }^{ t } \mc{B}(t-s) f(s,x,u(s,x)) \d s -  \int\limits_{ 0 }^{ \mc{T} } \mc{D}(t,\mc{T}-s)  f(s,x,u(s,x)) \d s.
\end{align}
From the inequality $(a+b)^2 \le 2 (a^2+b^2)$, for $a,b \in \mb{R}$, and H\"older's inequality, one can see that 
\begin{align} \label{Euip}  
&\mb{E} \left\| u(t,\cdot) \right\|_{L^2(\Omega)}^2  \nn \\
&\le 2 \mb{E} \left[ \int\limits_{ 0 }^{ t } \left\| \mc{B}(t-s) f(s,\cdot,u(s,\cdot)) \right\|_{L^2(\Omega)}  \d s \right]^2 + 2 \mb{E}    \left[ \int\limits_{ 0 }^{ \mc{T} } \left\| \mc{D}(t,\mc{T}-s) f(s,\cdot,u(s,\cdot)) \right\|_{L^2(\Omega)}  \d s \right]^2 \nn \\
&\le 2 t \int\limits_{ 0 }^{ t } \mb{E} \left\|  \mc{B}(t-s) f(s,\cdot,u(s,\cdot)) \right\|_{L^2(\Omega)} ^2 \d s + 2 \mc{T} \int\limits_{ 0 }^{ \mc{T} } \mb{E} \left\| \mc{D}(t,\mc{T}-s) f(s,\cdot,u(s,\cdot)) \right\|_{L^2(\Omega)}^2  \d s.
\end{align}   
Using Lemma \ref{lmop} and the fact that $f(s,x,u(s,x)) = \mf{K} u(s,x)$, one gets 
\begin{align*}   
\left\| \mc{B}(t-s) f(s,\cdot,u(s,\cdot)) \right\|_{L^2(\Omega)} ^2 &= \mf{K}^2 \left\| \mc{B}(t-s) u(s,\cdot) \right\|_{L^2(\Omega)} ^2 \le \mf{K}^2 \ms{M}^2_2 \left\| u(s,\cdot) \right\|_{L^2(\Omega)}^2 ,
\end{align*} 
and 
\begin{align*}   &\left\| \mc{D}(t,\mc{T}-s) f(s,\cdot,u(s,\cdot)) \right\|_{L^2(\Omega)} ^2 = \mf{K}^2 \left\| \mc{D}(t,\mc{T}-s) u(s,\cdot) \right\|_{L^2(\Omega)} ^2 
\le \mf{K}^2 \left( \dfrac{\ms{M}^2_2 \mc{T}^\al}{\ms{M}_1 (\mc{T}-s)^\al} \right)^2 \left\| u(s,\cdot) \right\|_{L^2(\Omega)}^2 .\end{align*}     
Hence, we deduce that
\begin{align*}   
\mb{E} \left\| u(t,\cdot) \right\|_{L^2(\Omega)}^2  
&\le 2 \mf{K}^2 \ms{M}^2_2 t \int\limits_{ 0 }^{ t } \mb{E} \left\| u(s,\cdot) \right\|_{L^2(\Omega)}^2  \d s + 2 \mf{K}^2 \left( \dfrac{\ms{M}^2_2 }{\ms{M}_1 } \right)^2 \mc{T} \int\limits_{ 0 }^{ \mc{T} } \dfrac{ \mc{T}^{2\al}}{ (\mc{T}-s)^{2\al}}\mb{E} \left\| u(s,\cdot) \right\|_{L^2(\Omega)}^2  \d s \\
&\le 2 \mf{K}^2 \ms{M}^2_2 \mc{T}^2 \left( 1+ \frac{ \ms{M}_2^2 }{  (1-2\al)\ms{M}_1^2 }\right)   \sup_{t \in [0,\mc{T}]} \mb{E} \left\| u(t,\cdot) \right\|_{L^2(\Omega)}^2 .
\end{align*}
Since $2 \mf{K}^2 \ms{M}^2_2 \mc{T}^2 \left( 1+ \frac{ \ms{M}_2^2 }{  (1-2\al)\ms{M}_1^2 }\right) = \frac{1}{2}$, we conclude that $u \equiv 0$. 

Now, we are ready to estimate the error $\left\| u^{n_1,n_2} -  u \right\|_{\mb{X}_\mc{T}}$. Applying the result in Section \ref{ss3} (see \eqref{opereq}), one has
\begin{align*}   
u^{n_1,n_2}(0,x) = \mc{A}(0) \varphi^{n_1,n_2}(x) -  \int\limits_{ 0 }^{ \mc{T} } \mc{D}(0,\mc{T}-s)  f(s,x,u^{n_1,n_2}(s,x)) \d s.
\end{align*} 
Since $\left\langle \varphi^{n_1,n_2}, \xi_{j_1,j_2} \right\rangle =0$ for $j_1 \ge n_1$ or $j_2 \ge n_2$ and $E_{\al,1} (0 )=1$, one can see that
$$\mc{A}(0) \varphi^{n_1,n_2}  = \sum_{ j_1=1 }^{ n_1-1 } \sum_{ j_2=1 }^{ n_2-1 }  \frac{\left\langle \varphi^{n_1,n_2}, \xi_{j_1,j_2} \right\rangle}{E_{\al,1} \big(- \lambda_{j_1,j_2} \mathcal{T}^\al \big)}  \xi_{j_1,j_2}.$$ It follows that  
\begin{align}  \label{cbm}  
&2\left\| u^{n_1,n_2}(0,\cdot) \right\|_{L^2(\Omega)}^2  \nn \\
&\ge \left\|\mc{A}(0) \varphi^{n_1,n_2} \right\|_{L^2(\Omega)}^2   -  2\left\| \int\limits_{ 0 }^{ \mc{T} }  \mc{D}(0,\mc{T}-s)  f(s,\cdot,u^{n_1,n_2}(s,\cdot))  \d s \right\|_{L^2(\Omega)}^2  \nn \\
&\ge \sum_{ j_1=1 }^{ n_1-1 } \sum_{ j_2=1 }^{ n_2-1 }  \dfrac{\left\langle \varphi^{n_1,n_2}, \xi_{j_1,j_2} \right\rangle^2}{E^2_{\al,1} \big(- \lambda_{j_1,j_2} \mathcal{T}^\al \big)}     - 2 \mf{K}^2 \Bigg[ \int\limits_{ 0 }^{ \mc{T} }  \left\| \mc{D}(0,\mc{T}-s) u^{n_1,n_2}(s,\cdot) \right\|_{L^2(\Omega)}   \d s \Bigg]^2 \nn \\
&\ge \sum_{ j_1=1 }^{ n_1-1 } \sum_{ j_2=1 }^{ n_2-1 } \dfrac{\lambda^2_{j_1,j_2}\mc{T}^{2\al}}{\ms{M}^2_2} \left\langle \varphi^{n_1,n_2}, \xi_{j_1,j_2} \right\rangle^2 - 2 \mf{K}^2 \mc{T} \int\limits_{ 0 }^{ \mc{T} }  \left\| \mc{D}(0,\mc{T}-s) u^{n_1,n_2}(s,\cdot) \right\|_{L^2(\Omega)}^2   \d s.
\end{align} 
By a similar calculation as in \eqref{einput}, we have
\begin{align}  \label{cmh}  
\mb{E} \left\langle \varphi^{n_1,n_2}, \xi_{j_1,j_2} \right\rangle^2 &= \mb{E} \Bigg( \frac{\pi^2}{n_1n_2} \sum_{ k_1=1 }^{ n_1 } \sum_{ k_2=1 }^{ n_2}  \Phi^{obs}_{k_1,k_2}  \xi_{j_1,j_2}(x_{k_1,k_2}) \Bigg)^2 = \frac{\pi^2}{n_1^{3/2} n_2^{3/2}}.
\end{align}      
On the other hand
\begin{align}  \label{cbb}  
2 \mf{K}^2 \mc{T} \int\limits_{ 0 }^{ \mc{T} } \mb{E} \left\| \mc{D}(0,\mc{T}-s) u^{n_1,n_2}(s,\cdot) \right\|_{L^2(\Omega)}^2   \d s &\le 2 \mf{K}^2 \left( \dfrac{\ms{M}^2_2 }{\ms{M}_1 } \right)^2 \mc{T} \int\limits_{ 0 }^{ \mc{T} } \dfrac{ \mc{T}^{2\al}}{ (\mc{T}-s)^{2\al}}\mb{E} \left\| u(s,\cdot) \right\|_{L^2(\Omega)}^2  \d s  \nn \\
&\le 2 \mf{K}^2  \mc{T}^2\frac{\ms{M}_2^4}{(1-2\al) \ms{M}_1^2} \sup_{t \in [0,\mc{T}]} \mb{E} \left\| u(t,\cdot) \right\|_{L^2(\Omega)}^2 . 
\end{align}
Combining \eqref{cbm}-\eqref{cbb}, we deduce that
\begin{align*}    
2\mb{E} \left\| u^{n_1,n_2}(0,\cdot) \right\|_{L^2(\Omega)}^2  \ge \dfrac{\pi^2 \mc{T}^{2\al}  }{\ms{M}_2^2  } \frac{\lambda^2_{n_1-1,n_2-1}}{n_1^{3/2} n_2^{3/2}} - 2 \mf{K}^2  \mc{T}^2\frac{\ms{M}_2^4}{(1-2\al) \ms{M}_1^2} \sup_{t \in [0,\mc{T}]} \mb{E} \left\| u(t,\cdot) \right\|_{L^2(\Omega)}^2 .
\end{align*} 
Since $2 \mf{K}^2  \mc{T}^2\frac{\ms{M}_2^4}{(1-2\al) \ms{M}_1^2} \le \frac{1}{2}$ and $\lambda_{j_1,j_2}=j_1^2+j_2^2$, we obtain 
\begin{align*}    
\frac{3}{2} \sup_{t \in [0,\mc{T}]} \mb{E} \left\| u^{n_1,n_2}(t,\cdot) \right\|_{L^2(\Omega)}^2  \ge \dfrac{\pi^2 \mc{T}^{2\al}  }{\ms{M}_2^2  } \frac{(n_1-1)^2+(n_2-1)^2}{n_1^{3/2} n_2^{3/2}}.
\end{align*} 
Using the fact that
\begin{align*}    
\left\| u^{n_1,n_2} - u \right\|_{\mb{X}_{\mc{T}}} = \left\| u^{n_1,n_2} \right\|_{\mb{X}_{\mc{T}}} = \sup_{t \in [0,\mc{T}]} \sqrt{\mb{E} \left\| u^{n_1,n_2}(t,\cdot) \right\|_{L^2(\Omega)}^2  } \ge \sqrt{ \sup_{t \in [0,\mc{T}]} \mb{E} \left\| u^{n_1,n_2}(t,\cdot) \right\|_{L^2(\Omega)}^2   },
\end{align*}
we conclude that $\left\| u^{n_1,n_2} - u \right\|_{\mb{X}_{\mc{T}}}$ tends to infinity as $n_1,n_2 \to \infty$.   
\end{example}
\subsection{Fourier truncated method and regularized solution} 
For $N_1,N_2$ are as in Lemma \ref{appp} and the function $g \in L^2(\Omega)$, we set
\begin{align*}    
\mb{A}^{N_1,N_2}(t) g &:= \sum_{ j_1=1 }^{ N_1 } \sum_{ j_2=1 }^{ N_2 } \Bigg( \dfrac{E_{\al,1} \big(- \lambda_{j_1,j_2} t^\al \big)}{E_{\al,1} \big(- \lambda_{j_1,j_2} \mathcal{T}^\al \big)} \left\langle g, \xi_{j_1,j_2} \right\rangle  \Bigg) \xi_{j_1,j_2}, \\
\mb{B}^{N_1,N_2}(t) g &:= \sum_{ j_1=1 }^{ N_1 } \sum_{ j_2=1 }^{ N_2 } \Big(  E_{\al,1} \big( -\lambda_{j_1,j_2} t^\al \big) \left\langle g, \xi_{j_1,j_2} \right\rangle \Big) \xi_{j_1,j_2},
\end{align*} 
and $\mb{D}^{N_1,N_2}(t,s)g:=\mb{A}^{N_1,N_2}(t) \mb{B}^{N_1,N_2}(s)g$, which are truncated series of $\mc{A}(t)g,\mc{B}(t)g$ and $\mc{D}(t,s)g$ respectively. 
Based on the approximate function for $\varphi$ constructed as in Lemma \ref{appp} (denote by $\widetilde{\varphi}^{N_1,N_2}$), we give the regularized solution as follows
\begin{align}  \label{unnre}  
\widetilde {\mf{u}}^{N_1,N_2}(t,x) &= \mb{A}^{N_1,N_2}(t) \varphi^{N_1,N_2}(x) + \int\limits_{ 0 }^{ t } \mb{B}^{N_1,N_2}(t-s) f(s,x,\widetilde {\mf{u}}^{N_1,N_2}(s,x)) \d s \nn \\
&-  \int\limits_{ 0 }^{ \mc{T} } \mb{D}^{N_1,N_2}(t,\mc{T}-s)  f(s,x,\widetilde {\mf{u}}^{N_1,N_2}(s,x)) \d s.
\end{align}

Before presenting the convergence rate between $\widetilde{\mf{u}}^{N_1,N_2}$ and the solution $u$ of \eqref{pr}  of the two-dimensional, we show the existence and uniqueness of the solution $\widetilde{\mf{u}}^{N_1,N_2}$:
\begin{lemma} \label{uniquelm}
Assume that $f$ satisfies the globally Lipschitz property, i.e. there exists a positive
constant $K$ such that 
\begin{align} \label{aspt}   
\left\| f(t,\cdot,u_1(t,\cdot)) - f(t,\cdot,u_2(t,\cdot)) \right\|_{{L}^2(\Omega) }  \le K \left\| u_1(t,\cdot) - u_2(t,\cdot) \right\|_{L^2(\Omega)} , \quad u_1,u_2 \in L^2(\Omega).
\end{align}
Assume further that $K \in \left( 0, \mc{Q}^{-1}_{\al,\mc{T}} \right) $  where
\begin{align} \label{ask}
\mc{Q}_{\al,\mc{T}} =  \dfrac{2 \ms{M}_2 \mc{T} \sqrt{2 \left( \ms{M}_1^2(1-\al)^2+\ms{M}_2^2 \right) }}{\ms{M}_1(1-\al)} .
\end{align}     
Then, the integral equation \eqref{unnre} has a unique solution $\widetilde{\mf{u}}^{N_1,N_2} \in \mb{X}_{\mc{T}}$.
\end{lemma}
\begin{proof}
Put
\begin{align*}    
\mf{F}(v(t,x))&:= \mb{A}^{N_1,N_2}(t) \varphi^{N_1,N_2}(x) + \int\limits_{ 0 }^{ t } \mb{B}^{N_1,N_2}(t-s) f(s,x,v(s,x)) \d s \\
&-  \int\limits_{ 0 }^{ \mc{T} } \mb{D}^{N_1,N_2}(t,\mc{T}-s)  f(s,x,v(s,x)) \d s.
\end{align*}
We will show that $\left\| \mf{F}(v_1)-\mf{F}(v_2) \right\|_{\mb{X}_\mc{T}}  \leq  \dfrac{K \mc{Q}_{\alpha,\mc{T}}}{2} \left\| v_1- v_2 \right\|_{\mb{X}_\mc{T}} $, which implies that $\mf{F}$ is a contraction and thus the equation \eqref{unnre} has a unique solution. We first have
\begin{align*}    
\left\| \mf{F} (v_1(t,\cdot)) - \mf{F} (v_2(t,\cdot)) \right\|_{L^2(\Omega)}  &\le \int\limits_{ 0 }^{ t } \left\| \mb{B}^{N_1,N_2}(t-s) \big( f(s,\cdot,v_1(s,\cdot)) - f(s,\cdot,v_2(s,\cdot))\big)   \right\|_{L^2(\Omega)}    \d s \\
&+  \int\limits_{ 0 }^{ \mc{T} } \left\| \mb{D}^{N_1,N_2}(t,\mc{T}-s)  \big( f(s,\cdot,v_1(s,\cdot)) - f(s,\cdot,v_2(s,\cdot))\big)   \right\|_{L^2(\Omega)}  \d s.
\end{align*}    
Using Lemma \ref{lmop} and  assumption \eqref{aspt}, we get
\begin{align*}    
\left\| \mf{F} (v_1(t,\cdot)) - \mf{F} (v_2(t,\cdot)) \right\|_{L^2(\Omega)}  &\le K \ms{M}_2  \int\limits_{ 0 }^{ t } \left\| v_1(s,\cdot) - v_2(s,\cdot)  \right\|_{L^2(\Omega)}  \d s \\
&+ K \frac{\ms{M}_2 ^2}{\ms{M}_1 } \mc{T}^{\al} \int\limits_{ 0 }^{ \mc{T} } (\mc{T}-s)^{-\al}\left\| v_1(s,\cdot) - v_2(s,\cdot)  \right\|_{L^2(\Omega)}  \d s.
\end{align*} 
By H\"older's inequality and $(a+b)^2 \le 2(a^2+b^2)$, for $a,b \in \mb{R}$, we obtain
\begin{align*}    
\mb{E} \big\| \mf{F} (v_1(t,\cdot)) - \mf{F} &(v_2(t,\cdot)) \big\|_{L^2(\Omega)}^2  \le 2 K^2 \ms{M}_2^2 t \int\limits_{ 0 }^{ t } \mb{E} \left\| v_1(t,\cdot) - v_2(t,\cdot)  \right\|_{L^2(\Omega)}^2  \d s \\
&+ 2 K^2 \dfrac{\ms{M}_2^4}{\ms{M}_1^2} \mc{T}^{2 \al} \int\limits_{ 0 }^{ \mc{T} } (\mc{T}-s)^{-\al} \d s  \int\limits_{ 0 }^{ \mc{T} } (\mc{T}-s)^{-\al} \mb{E} \left\| v_1(t,\cdot) - v_2(t,\cdot)  \right\|_{L^2(\Omega)}^2  \d s.
\end{align*}  
It follows that 
\begin{align*}    
\mb{E} \big\| \mf{F} (v_1(t,\cdot)) - \mf{F} (v_2(t,\cdot)) \big\|_{L^2(\Omega)}^2  &\le 2 K^2 \ms{M}_2^2 \mc{T}^2 \sup_{t \in [0,\mc{T}]} \mb{E} \left\| v_1(t,\cdot) - v_2(t,\cdot)  \right\|_{L^2(\Omega)}^2  \\
&+ 2 K^2 \dfrac{\ms{M}_2^4}{\ms{M}_1^2}  \frac{\mc{T}^{2}}{(1-\al)^2} \sup_{t \in [0,\mc{T}]} \mb{E} \left\| v_1(t,\cdot) - v_2(t,\cdot)  \right\|_{L^2(\Omega)}^2 .
\end{align*}
Thus
\begin{align*}    
\left\| \mf{F}(v_1) - \mf{F}(v_2) \right\|^2_{\mb{X}_\mc{T}} &\le 2K^2 \ms{M}_2^2 \mc{T}^2 \left( 1+ \dfrac{\ms{M}_2^2}{\ms{M}_1^2(1-\al)^2}  \right) \sup_{t \in [0,\mc{T}]} \mb{E} \left\| v_1(t,\cdot) - v_2(t,\cdot)  \right\|_{L^2(\Omega)}^2 \\
&= \left( \dfrac{K \mc{Q}_{\alpha,\mc{T}}}{2} \right)^2 \sup_{t \in [0,\mc{T}]} \mb{E} \left\| v_1(t,\cdot) - v_2(t,\cdot)  \right\|_{L^2(\Omega)}^2 .
\end{align*} 
We conclude that $\left\| \mf{F}(v_1) - \mf{F}(v_2) \right\|_{\mb{X}_\mc{T}} \leq  \dfrac{K \mc{Q}_{\alpha,\mc{T}}}{2}  \left\| v_1 - v_2 \right\|_{\mb{X}_\mc{T}}$. This completes the proof.   
\end{proof}
Note in the above result we could assume $K \in \left( 0, 2 \mc{Q}^{-1}_{\al,\mc{T}} \right) $. However $K \in \left( 0, \mc{Q}^{-1}_{\al,\mc{T}} \right) $ is needed in our next result.

\subsection{Convergence result}
Now, we are ready to state the main result of the present section in the following theorem:
\begin{theorem} Let $N_1=N_1(n_1)$, $N_2=N_2(n_2)$ be natural numbers less than $n_1,n_2$ respectively and satisfy
\begin{align} \label{gas}
\lim_{n_1 \to \infty} N_1 = \lim_{n_2 \to \infty} N_2 = \infty, \quad  {\lim_{n_1, n_2 \to \infty} } \dfrac{N_1N_2 \lambda_{N_1,N_2}^2}{n_1n_2}=0.
\end{align}   Assume the conditions of Lemma \ref{appp}, Lemma \ref{uniquelm} hold and $u \in L^\infty(0,\mc{T};H^\theta(\Omega))$. Then, the error $\left\| \widetilde {\mf{u}} ^{N_1,N_2} - u \right\|^2_{\mb{X}_\mc{T}}$ is of order
\begin{align*}    
 {\max \left\{ \dfrac{N_1N_2 }{n_1n_2}(N_1^4+N_2^4),N_1^{-2 \theta},N_2^{-2 \theta} \right\}}.
\end{align*} 
 {
\begin{remark}
If we choose 
\begin{align*}           
(N_1,N_2) = (\lfloor n_1^{1/5} \rfloor, \lfloor n_2^{1/5} \rfloor) \quad \text{or} \quad (N_1,N_2) = (\lfloor \log n_1 \rfloor, \lfloor \log n_2 \rfloor),
\end{align*} 
where we denote $\lfloor p \rfloor$ the greatest natural number less than $p$,  then $N_1,N_2$ satisfy \eqref{gas}. 
\end{remark} 
}
\begin{proof} Let 
\begin{align}   \label{vnnre}  v^{N_1,N_2}(t,x) = \mb{A}^{N_1,N_2}(t) \varphi(x) + \int\limits_{ 0 }^{ t } \mb{B}^{N_1,N_2}&(t-s) f(s,x,u(s,x)) \d s \nn \\
&-  \int\limits_{ 0 }^{ \mc{T} } \mb{D}^{N_1,N_2}(t,\mc{T}-s)  f(s,x,u(s,x)) \d s,
\end{align}
which is the truncated series of $u$. Then, we have
\begin{align*}    
\frac{1}{2}\left\| \widetilde {\mf{u}} ^{N_1,N_2}(t,\cdot) - u(t,\cdot) \right\|_{L^2(\Omega)}^2  \le  \left\| \widetilde {\mf{u}} ^{N_1,N_2}(t,\cdot) - v^{N_1,N_2}(t,\cdot) \right\|_{L^2(\Omega)}^2  +  \left\| v^{N_1,N_2}(t,\cdot) - u(t,\cdot) \right\|_{L^2(\Omega)}^2 .
\end{align*}  
\textbf{Step 1} (Estimating the error between $\widetilde {\mf{u}}^{N_1,N_2}$ and $v^{N_1,N_2}$). From \eqref{unnre} and \eqref{vnnre}, one can see that
\begin{align} \label{lastm}   
\big\| \widetilde {\mf{u}}^{N_1,N_2}(t,\cdot)- v^{N_1,N_2}(t,\cdot) \big\|_{L^2(\Omega)}  &\le \left\| \mb{A}^{N_1,N_2}(t) \left( \widetilde \varphi^{N_1,N_2} - \varphi \right)  \right\|_{L^2(\Omega)}  \nn \\
&+  \int\limits_{ 0 }^{ t }  \left\| \mb{B}^{N_1,N_2}(t-s) \big( f(s,\cdot,\widetilde{ \mf{u}}^{N_1,N_2}(s,\cdot)) -  f(s,\cdot,u(s,\cdot))  
\big)   \right\|_{L^2(\Omega)}  \d s \nn \\
&+ \int\limits_{ 0 }^{ \mc{T} }  \left\| \mb{D}^{N_1,N_2}(t,\mc{T}-s) \big( f(s,\cdot,\widetilde{ \mf{u}}^{N_1,N_2}(s,\cdot)) -  f(s,\cdot,u(s,\cdot))  
\big)  \right\|_{L^2(\Omega)}  \d s \nn \\
&=: \mc{I}_1 + \mc{I}_2 + \mc{I}_3.
\end{align} 
For $g \in L^2(\Omega)$, we have 
\begin{align*}    
\left\| \mb{A}^{N_1,N_2}(t) g \right\|_{L^2(\Omega)}^2  &\le \sum_{ j_1=1 }^{ N_1 } \sum_{ j_2=1 }^{ N_2 } \left( \frac{\ms{M}_2 (1+ \lambda_{j_1,j_2} \mathcal{T}^\al )}{\ms{M}_1} \right)^2   \left\langle g, \xi_{j_1,j_2} \right\rangle^2 \\
&\le \left( \frac{\ms{M}_2 (1+ \lambda_{N_1,N_2} \mathcal{T}^\al )}{\ms{M}_1} \right)^2 \sum_{ j_1=1 }^{ N_1 } \sum_{ j_2=1 }^{ N_2 } \left\langle g, \xi_{j_1,j_2} \right\rangle^2,
\end{align*}
and recall that (see \eqref{4e} and \eqref{e1})
\begin{align*}    
\mb{E} \mc{E}_1 = \mb{E} \Bigg[ \sum_{ j_1=1 }^{ N_1 } \sum_{ j_2=1 }^{ N_2 } \left( \left\langle \widetilde  {\varphi}^{N_1,N_2}, \xi_{j_1,j_2} \right\rangle - \varphi_{j_1,j_2} \right) ^2 \Bigg] \le \dfrac{\left( 2 \pi^2 \varepsilon^2_{\max} +  2 \mc{C}_0^2  \left\| \varphi \right\|^2_{H^\theta (\Omega)} \right) N_1 N_2 }{n_1 n_2},
\end{align*}    
Hence
\begin{align}   \label{lasth}  
\mb{E} \mc{I}_1^2 \le  \left( \frac{\ms{M}_2 (1+ \lambda_{N_1,N_2} \mathcal{T}^\al )}{\ms{M}_1} \right)^2 \mb{E} \mc{E}_1 \le \left( \frac{\ms{M}_2 (1+ \lambda_{N_1,N_2} \mathcal{T}^\al )}{\ms{M}_1} \right)^2 \dfrac{\left( 2 \pi^2 \varepsilon^2_{\max} +  2 \mc{C}_0^2  \left\| \varphi \right\|^2_{H^\theta (\Omega)} \right) N_1 N_2 }{n_1 n_2}.
\end{align}
For the term $\mc{I}_2+\mc{I}_3$, by a similar technique as in the proof of Lemma \ref{uniquelm}, one arrives at
\begin{align}   \label{lastb}  
\mb{E} \left( \mc{I}_2+ \mc{I}_3 \right)^2 \le 2K^2 \ms{M}_2^2 \mc{T}^2 \left( 1+ \dfrac{\ms{M}_2^2}{\ms{M}_1^2(1-\al)^2}  \right) \sup_{t \in [0,\mc{T}]} \mb{E} \left\| \widetilde{ \mf{u}}^{N_1,N_2}(t,\cdot) - u(t,\cdot)  \right\|_{L^2(\Omega)}^2 .
\end{align} 
Combining \eqref{lastm}-\eqref{lastb}, one obtains
\begin{align*}    
\mb{E} \big\| \widetilde {\mf{u}}^{N_1,N_2}(t,\cdot)&- v^{N_1,N_2}(t,\cdot) \big\|^2  \\
&\le 4 \left( \frac{\ms{M}_2 (1+ \lambda_{N_1,N_2} \mathcal{T}^\al )}{\ms{M}_1} \right)^2 \dfrac{\left( \pi^2 \varepsilon^2_{\max} +  \mc{C}_0^2  \left\| \varphi \right\|^2_{H^\theta (\Omega)} \right) N_1 N_2 }{n_1 n_2} \\
&+4K^2 \ms{M}_2^2 \mc{T}^2 \left( 1+ \dfrac{\ms{M}_2^2}{\ms{M}_1^2(1-\al)^2}  \right)  \left\| \widetilde{ \mf{u}}^{N_1,N_2} - u \right\|^2_{\mb{X}_{\mc{T}}}.
\end{align*}  
\textbf{Step 2} (Estimating the error between $v^{N_1,N_2}$ and $u$). From \eqref{opereq} and \eqref{vnnre}, one can see that
\begin{align*}
\left\| v^{N_1,N_2}(t,\cdot) - u(t,\cdot) \right\|_{L^2(\Omega)}^2  =    
\sum_{ j_1=1 }^{ N_1 } \sum_{ j_2=N_2+1 }^{ \infty } u_{j_1,j_2}^2(t) + \sum_{ j_1=N_1+1 }^{ \infty } \sum_{ j_2=1 }^{ N_2 } u_{j_1,j_2}^2(t)  + \sum_{ j_1=N_1+1 }^{ \infty } \sum_{ j_2=N_2+1 }^{ \infty } u_{j_1,j_2}^2(t).
\end{align*} 
The quantity above can be estimated in exactly the same way as in Part B in the proof of Lemma \ref{appp}. In this way, one gets
\begin{align*}    
\left\| v^{N_1,N_2} - u \right\|_{L^2(\Omega)}^2  &\le 2 \left[ (N_1+1)^{-2 \theta} + (N_2+1)^{-2 \theta} \right] \left\| u(t,\cdot) \right\|^2_{H^\theta(\Omega)} \\
&\le 2 \left[ (N_1+1)^{-2 \theta} + (N_2+1)^{-2 \theta} \right] \left\| u \right\|^2_{L^\infty(0,\mc{T};H^\theta(\Omega))}. 
\end{align*}
Now, using the results of the two steps, we deduce that
\begin{align*}    
\frac{1}{2}\mb{E} \left\| \widetilde {\mf{u}} ^{N_1,N_2}(t,\cdot) - u(t,\cdot) \right\|_{L^2(\Omega)}^2  &\le 4 \left( \frac{\ms{M}_2 (1+ \lambda_{N_1,N_2} \mathcal{T}^\al )}{\ms{M}_1} \right)^2 \dfrac{\left( \pi^2 \varepsilon^2_{\max} +  \mc{C}_0^2  \left\| \varphi \right\|^2_{H^\theta (\Omega)} \right) N_1 N_2 }{n_1 n_2} \\
&+4K^2 \ms{M}_2^2 \mc{T}^2 \left( 1+ \dfrac{\ms{M}_2^2}{\ms{M}_1^2(1-\al)^2}  \right) \left\| \widetilde{ \mf{u}}^{N_1,N_2} - u \right\|^2_{\mb{X}_{\mc{T}}} \\
&+ 2 \left[ (N_1+1)^{-2 \theta} + (N_2+1)^{-2 \theta} \right] \left\| u \right\|^2_{L^\infty(0,\mc{T};H^\theta(\Omega))},
\end{align*} 
which gives us
\begin{align*}    
 \frac{ 1 - K^2 \mc{Q}^2_{\al, \mc{T}}}{2} \left\| \widetilde{ \mf{u}}^{N_1,N_2} - u \right\|^2_{\mb{X}_{\mc{T}}} \le &4\left( \frac{\ms{M}_2 (1+ \lambda_{N_1,N_2} \mathcal{T}^\al )}{\ms{M}_1} \right)^2 \dfrac{\left( \pi^2 \varepsilon^2_{\max} +  \mc{C}_0^2  \left\| \varphi \right\|^2_{H^\theta (\Omega)} \right) N_1 N_2 }{n_1 n_2} \\
&+ 2 \left[ (N_1+1)^{-2 \theta} + (N_2+1)^{-2 \theta} \right] \left\| u \right\|^2_{L^\infty(0,\mc{T};H^\theta(\Omega))}.
\end{align*} 
Hence, we conclude that $\left\| \widetilde{ \mf{u}}^{N_1,N_2} - u \right\|^2_{\mb{X}_{\mc{T}}}$ is of order
\begin{align*}    
	 {\max \left\{ \dfrac{N_1N_2 }{n_1n_2}(N_1^4+N_2^4),N_1^{-2 \theta},N_2^{-2 \theta} \right\}}.
\end{align*}      
\end{proof}  
\end{theorem}
\section{Backward problem in the  {multi-dimensional case}} \label{s5}
Based on Subsection \ref{subsec33}, we claim that the multi-dimensional backward problem \eqref{pr} with discrete data is not well-posed. Thus, a regularized method is required to construct a stable approximate solution. To do this, in next subsection, we establish an approximation for the final data $\varphi$.
\subsection{Estimator for $\varphi$ in the  {multi-dimensional case}}
For any positive constant   $\gamma_{\mbf{n}}=\gamma_{n_1,n_2,\dots,n_d}$ depending on $\mbf{n}=(n_1,n_2,\dots,n_d)$,  we define
\begin{align}  \label{wgm}  
\mb{W}_{\gamma_{\mbf{n}}} = \left\{ \mbf{j} = (j_1,j_2,\dots,j_d) \in \mb{N}^d: |\mbf{j}|^2 = \sum_{ i=1 }^{ d } j_i^2 \le \gamma_{\mbf{n}} \right\}.
\end{align} 
 {For $\gamma_{\mbf{n}}$ satisfying $\lim_{|\mbf{n}| \to \infty} \gamma_{\mbf{n}}=\infty$, we define an approximation for $\varphi$ as follows 
\begin{align}   \label{estmul}  \widehat{\varphi}^{\gamma_{\mbf{n}} } = \sum_{\mbf{j} \in \mb{W}_{\gamma_{\mbf{n}}}} \Bigg[
\dfrac{\pi^d}{\prod_{i=1}^d n_i} \sum_{ k_1=1 }^{ n_1 } \sum_{ k_2=1 }^{ n_2 } \dots \sum_{ k_d=1 }^{ n_d } \Phi^{obs}_{\mbf{k}} \xi_{\mbf{j}}(x_{\bf{k}})
\Bigg] \xi_{\mbf{j}}.
\end{align} }
\begin{theorem}  \label{maintheomul}
Let $\mu=(\mu_1,\mu_2,\dots,\mu_d) \in \mb{R}^d$ in which $\mu_k > \frac{1}{2}$ for any $k=1,2,\dots,d$ and $\mu_{\circ} \in \mb{R}^+$ such that $\mu_{\circ} \ge d \max \left( \mu_1,\mu_2,\dots, \mu_d \right) $. Then we have
\begin{itemize}
\item [(a)] \textbf{Error estimate in $L^2(\Omega)$} (see Theorem 2.1 of \cite{Kir}). If $\varphi \in H^{\mu_\circ}(\Omega)$ then
	\begin{align*}    
	\mb{E} \left\| \widehat{\varphi}^{\gamma_{\mbf{n}} }  - \varphi \right\|^2_{{L}^2(\Omega)} \le  {\overline{C} (\mu,\varphi)} \gamma_{\mbf{n}}^{d/2} \prod_{i=1}^d n_i^{-4\mu_i} + 4 \gamma_{\mbf{n}}^{-\mu_\circ} \left\| \varphi \right\|^2_{H^{\mu_0}(\Omega)},
	\end{align*} 
	where 
	\begin{align*}    
 {\overline{C} (\mu,\varphi)} = 8 \pi^d \varepsilon_{\max}^2 \dfrac{2 \pi^{\frac{d}{2}}}{d \Gamma \left( \frac{d}{2} \right)} + \dfrac{16~  {{C}^2 (\mu)} \pi^{\frac{d}{2}}}{d \Gamma \left( \frac{d}{2} \right) } \left\| \varphi \right\|^2_{H^{\mu_\circ}(\Omega)},
	\end{align*}
	with 
	\begin{align*}    
 {C (\mu)} = d^{\frac{-\max(\mu_1,\dots,\mu_d)}{2}}  \sum_{\mbf{l} \in \mb{N}^d, |\mbf{l}| \neq 0 } \prod_{i=1}^d (2l_i-1)^{-2 \mu_i} .
	\end{align*}     
\item [(b)] \textbf{Error estimate in $H^\sigma(\Omega)$}. If there exists a constant $\sigma>0$ such that $\varphi \in H^{\mu_\circ+\sigma}(\Omega)$ then
\begin{align*}    
\mb{E} \big\| \widehat{\varphi}^{\gamma_{\mbf{n}}} -\varphi \big\|^2_{H^{\sigma}(\Omega)} &\le  {\overline{C} (\mu,\varphi)} \frac{\gamma_{\mbf{n}}^{\sigma + \frac{d}{2}}}{4}  \prod_{i=1}^{d} (n_i)^{-4 \mu_i} +  \gamma_{\mbf{n}}^{-\mu_\circ} \left\| \varphi \right\|^2_{H^{\mu_\circ+ \sigma}(\Omega)}  .
\end{align*} 
\end{itemize}
\end{theorem}
\begin{proof} [Proof of Part (b)] 
First, by Lemma 2.3 of \cite{Kir}, we have
\begin{align*}    
\varphi_{\mbf{j}} = \dfrac{\pi^d}{\prod_{i=1}^d n_i} \sum_{ k_1=1 }^{ n_1 } \sum_{ k_2=1 }^{ n_2 } \dots \sum_{ k_d=1 }^{ n_d } \varphi(x_{\mbf{k}}) \xi_{\mbf{j}}(x_{\bf{k}}) - \sum_{ \mbf{p} = 2 \mbf{l} \cdot \mbf{n} \pm \mbf{j}  }^{ \mbf{j} \in \mb{N}^d, l_1^2+\dots+l_d^2 \neq 0 } \varphi_{\mbf{p} }. 
\end{align*} 
Thus, we get 
\begin{align*}    
\widehat{\varphi}^{\gamma_{\mbf{n}}}&(x) -\varphi(x)  \\ &= \sum_{ \mbf{j} \in \mb{W}_{\gamma_{\mbf{n}}} } \Bigg[  
\dfrac{\pi^d}{\prod_{i=1}^d n_i} \sum_{ k_1=1 }^{ n_1 } \sum_{ k_2=1 }^{ n_2 } \dots \sum_{ k_d=1 }^{ n_d } \Phi^{obs}_{\mbf{k}} \xi_{\mbf{j}}(x_{\bf{k}}) - \varphi_{\mbf{j}} 
\Bigg] \xi_{\mbf{j}}(x) - \sum_{ \mbf{j} \notin \mb{W}_{\gamma_{\mbf{n}}} } \varphi_{\mbf{j}} \xi_{\mbf{j}}(x) \\
&= \sum_{ \mbf{j} \in \mb{W}_{\gamma_{\mbf{n}}} } \Bigg[ 
\dfrac{\pi^d}{\prod_{i=1}^d n_i} \sum_{ k_1=1 }^{ n_1 } \sum_{ k_2=1 }^{ n_2 } \dots \sum_{ k_d=1 }^{ n_d } \varepsilon_{\mbf{k}} \ms{W}_{\mbf{k}} \varphi_{\mbf{j}}(x_{\bf{k}}) + \sum_{ \mbf{p} = 2 \mbf{l} \cdot \mbf{n} \pm \mbf{j}  }^{ \mbf{j} \in \mb{N}^d, l_1^2+\dots+l_d^2 \neq 0 } \varphi_{\mbf{p} }
\Bigg] \xi_{\mbf{j}}(x) - \sum_{ \mbf{j} \notin \mb{W}_{\gamma_{\mbf{n}}} } \varphi_{\mbf{j}} \xi_{\mbf{j}}(x).
\end{align*} 
It follows that
\begin{align}   \label{ccc} 
&~\quad \big\| \widehat{\varphi}^{\gamma_{\mbf{n}}} -\varphi \big\|^2_{H^{\sigma}(\Omega)} \nn \\
&= \sum_{ \mbf{j} \in \mb{W}_{\gamma_{\mbf{n}}} } \lambda_{\mbf{j} }^{\sigma} \Bigg[ 
\dfrac{\pi^d}{\prod_{i=1}^d n_i} \sum_{ k_1=1 }^{ n_1 } \sum_{ k_2=1 }^{ n_2 } \dots \sum_{ k_d=1 }^{ n_d } \varepsilon_{\mbf{k}} \ms{W}_{\mbf{k}} \varphi_{\mbf{j}}(x_{\bf{k}}) + \sum_{ \mbf{p} = 2 \mbf{l} \cdot \mbf{n} \pm \mbf{j}  }^{ \mbf{j} \in \mb{N}^d, l_1^2+\dots+l_d^2 \neq 0 } \varphi_{\mbf{p} }
\Bigg]^2 + \sum_{ \mbf{j} \notin \mb{W}_{\gamma_{\mbf{n}}} }  \lambda_{\mbf{j} }^{\sigma} \varphi_{\mbf{j}}^2 \nn \\
&\le 2 \sum_{ \mbf{j} \in \mb{W}_{\gamma_{\mbf{n}}} } \lambda_{\mbf{j} }^{\sigma} \Bigg[ 
\dfrac{\pi^d}{\prod_{i=1}^d n_i} \sum_{ k_1=1 }^{ n_1 } \sum_{ k_2=1 }^{ n_2 } \dots \sum_{ k_d=1 }^{ n_d } \varepsilon_{\mbf{k}} \ms{W}_{\mbf{k}} \varphi_{\mbf{j}}(x_{\bf{k}})
\Bigg]^2 + 2\sum_{ \mbf{j} \in \mb{W}_{\gamma_{\mbf{n}}} } \lambda_{\mbf{j} }^{\sigma} \Bigg[ \sum_{ \mbf{p} = 2 \mbf{l} \cdot \mbf{n} \pm \mbf{j}  }^{ \mbf{j} \in \mb{N}^d, l_1^2+\dots+l_d^2 \neq 0 } \varphi_{\mbf{p} } \Bigg]^2 + \sum_{ \mbf{j} \notin \mb{W}_{\gamma_{\mbf{n}}} }  \lambda_{\mbf{j} }^{\sigma} \varphi_{\mbf{j}}^2 \nn \\
&=: \ms{C}_{I} + \ms{C}_{II} + \ms{C}_{III}. 
\end{align} 
From Lemma 2.2 of \cite{Kir} and the property of  $\ms{W}_{\mbf{k}}$, the first term can be estimated as follows
\begin{align*}    
\mb{E} \ms{C}_{I} &\le  2 \sum_{ \mbf{j} \in \mb{W}_{\gamma_{\mbf{n}}} } \gamma_{\mbf{n}} ^{\sigma}  
\dfrac{\pi^{2d}}{ \left( \prod_{i=1}^d n_i \right)^2  } \mb{E} \Bigg[ \sum_{ k_1=1 }^{ n_1 } \sum_{ k_2=1 }^{ n_2 } \dots \sum_{ k_d=1 }^{ n_d } \varepsilon_{\mbf{k}} \ms{W}_{\mbf{k}} \varphi_{\mbf{j}}(x_{\bf{k}})
\Bigg]^2 \\
&\le  \dfrac{2 \pi^{2d}}{ \left( \prod_{i=1}^d n_i \right)^2  } \gamma_{\mbf{n}} ^{\sigma} \varepsilon^2_{\max} \text{card} \left( \mb{W}_{\gamma_{\mbf{n}}} \right) \sum_{ k_1=1 }^{ n_1 } \sum_{ k_2=1 }^{ n_2 } \dots \sum_{ k_d=1 }^{ n_d } \varphi^2_{\mbf{j}}(x_{\bf{k}})  \\
&= \dfrac{2 \pi^{d}}{ \prod_{i=1}^d n_i  } \gamma_{\mbf{n}} ^{\sigma} \varepsilon^2_{\max} \text{card} \left( \mb{W}_{\gamma_{\mbf{n}}} \right) .
\end{align*} 
In addition, using the inequality (2.30) of \cite{Kir}  that
$\text{card} \left( \mb{W}_{\gamma_{\mbf{n}}} \right) \le \dfrac{2 \pi^{\frac{d}{2}}} {d \Gamma \big( \pi^{\frac{d}{2}} \big)} \gamma_{\mbf{n}}^{\frac{d}{2}}$, 
we deduce that
\begin{align}   \label{c1} 
\mb{E} \ms{C}_{I}  \le  \varepsilon^2_{\max} \dfrac{4 \pi^{\frac{3d}{2}}} {d \Gamma \big( \pi^{\frac{d}{2}} \big) \prod_{i=1}^d n_i} \gamma_{\mbf{n}}^{\sigma+\frac{d}{2}}.
\end{align}  
From (2.37) of \cite{Kir}, the second term can be estimated as follows
\begin{align}   \label{c2} 
\ms{C}_{II} &=2\sum_{ \mbf{j} \in \mb{W}_{\gamma_{\mbf{n}}} } \lambda_{\mbf{j} }^{\sigma} \Bigg[ \sum_{ \mbf{p} = 2 \mbf{l} \cdot \mbf{n} \pm \mbf{j}  }^{ \mbf{j} \in \mb{N}^d, l_1^2+\dots+l_d^2 \neq 0 } \varphi_{\mbf{p} } \Bigg]^2 \nn \\
&\le 2  \gamma_{\mbf{n}} ^{\sigma}  {C^2 (\mu)} \left\| \varphi \right\|^2_{H^{\mu_{\circ}}(\Omega)}  \prod_{i=1}^{d} (n_i)^{-4 \mu_i} \text{card} (\mb{W}_{\gamma_{\mbf{n}} })  \le 2  \gamma_{\mbf{n}} ^{\sigma+\frac{d}{2}}  {C^2 (\mu)} \left\| \varphi \right\|^2_{H^{\mu_{\circ}}(\Omega)} \dfrac{2 \pi^{\frac{d}{2}}} {d \Gamma \big( \pi^{\frac{d}{2}} \big)} \prod_{i=1}^{d} (n_i)^{-4 \mu_i}  .
\end{align}  
For the last term, it is clear that
\begin{align}  \label{c3}  
\ms{C}_{III} = \sum_{ \mbf{j} \notin \mb{W}_{\gamma_{\mbf{n}}} } \lambda_{\mbf{j} }^{-\mu_\circ} \lambda_{\mbf{j} }^{\mu_\circ+ \sigma} \varphi_{\mbf{j}}^2 \le \gamma_{\mbf{n}}^{-\mu_\circ} \left\| \varphi \right\|^2_{H^{\mu_\circ+ \sigma}(\Omega)}. 
\end{align}
Combining \eqref{ccc}-\eqref{c3}, we conclude that
\begin{align*}    
\mb{E} \big\| \widehat{\varphi}^{\gamma_{\mbf{n}}} -\varphi \big\|^2_{H^{\sigma}(\Omega)} &\le \varepsilon^2_{\max} \dfrac{4 \pi^{\frac{3d}{2}}} {d \Gamma \big( \pi^{\frac{d}{2}} \big) \prod_{i=1}^d n_i} \gamma_{\mbf{n}}^{\sigma+\frac{d}{2}} +  \gamma_{\mbf{n}}^{-\mu_\circ} \left\| \varphi \right\|^2_{H^{\mu_\circ+ \sigma}(\Omega)} \\
&+ 2  \gamma_{\mbf{n}} ^{\sigma+\frac{d}{2}}  {C^2 (\mu)} \left\| \varphi \right\|^2_{H^{\mu_{\circ}}(\Omega)} \dfrac{2 \pi^{\frac{d}{2}}} {d \Gamma \big( \pi^{\frac{d}{2}} \big)} \prod_{i=1}^{d} (n_i)^{-4 \mu_i} \\
&\le  {\overline{C} (\mu,\varphi)} \frac{\gamma_{\mbf{n}}^{\sigma + \frac{d}{2}}}{4}  \prod_{i=1}^{d} (n_i)^{-4 \mu_i} +  \gamma_{\mbf{n}}^{-\mu_\circ} \left\| \varphi \right\|^2_{H^{\mu_\circ+ \sigma}(\Omega)},
\end{align*}     
which shows that $\mb{E} \big\| \widehat{\varphi}^{\gamma_{\mbf{n}}} -\varphi \big\|^2_{H^{\sigma}(\Omega)}$ is of order 
$   
\max \left(  \gamma_{\mbf{n}}^{\sigma+\frac{d}{2}} \prod_{i=1}^{d} (n_i)^{-4 \mu_i}, \gamma_{\mbf{n}}^{-\mu_\circ} \right).
$  
\end{proof}
\subsection{Quasi-boundary value method and regularized solution}

Using the estimator for $\varphi$ as in \eqref{estmul}, we give the regularized problem as follows
\begin{align} \label{prre} \begin{cases}   \dfrac{\partial}{\partial t} \widehat{\mf{u}}^{\gamma_{\mbf{n}},\vartheta_{\mbf{n}}  } (t,x) + \dfrac{\partial^{1-\alpha}}{\partial t}  \Delta \widehat{\mf{u}}^{\gamma_{\mbf{n}},\vartheta_{\mbf{n}}  }  (t,x) = f \left(t,x,\widehat{\mf{u}}^{\gamma_{\mbf{n}},\vartheta_{\mbf{n}}  } (t,x) \right), &(t,x)  \in (0,\mc{T})  \times  \Omega, \\
\widehat{\mf{u}}^{\gamma_{\mbf{n}},\vartheta_{\mbf{n}}  } (t,x) = 0, &(t,x) \in (0,\mc{T})  \times \partial \Omega, \\
\widehat{\mf{u}}^{\gamma_{\mbf{n}},\vartheta_{\mbf{n}}  } (\mc{T},x) + \vartheta_{\mbf{n}} \widehat{\mf{u}}^{\gamma_{\mbf{n}},\vartheta_{\mbf{n}}  } (0,x)= \widehat {\varphi}^{\gamma_{\mbf{n}} }(x), &x \in \Omega,
\end{cases}
\end{align} 
where $\vartheta_{\mbf{n}}>0$ depends on $\mbf{n}=(n_1,n_2,\dots,n_d)$ and satisfies
$   
\lim_{|\mbf{n}| \to \infty} \vartheta_{\mbf{n}} = 0
$, $\gamma_{\mbf{n}} $ fulfills Theorem \ref{maintheomul}. Here, the basic idea is to replace the final value data $\varphi$ by its approximation $\widehat {\varphi}_{\gamma_{\mbf{n}} }$ and  add the quantity $\vartheta_{\mbf{n}} \widehat{\mf{u}}^{\gamma_{\mbf{n}},\vartheta_{\mbf{n}}  } (0,x)$ into the left-hand side of the last equation. Then, using the result \eqref{ummm}, one has
\begin{align*}  
\widehat{\mf{u}}_{\mbf{j}}^{\gamma_{\mbf{n}},\vartheta_{\mbf{n}}  }  (t) = E_{\al,1} \big(- \lambda_{\mbf{j}} t^\al \big) \widehat{\mf{u}}_{\mbf{j}}^{\gamma_{\mbf{n}},\vartheta_{\mbf{n}}  }  (0)  + \int\limits_{ 0 }^{ t } E_{\al,1} \big( -\lambda_{\mbf{j}} (t-s)^\al \big)  f_{\mbf{j}}(\widehat{\mf{u}}^{\gamma_{\mbf{n}},\vartheta_{\mbf{n}}  } )(s) \d s,
\end{align*}
where we denote $\widehat{\mf{u}}_{\mbf{j}}^{\gamma_{\mbf{n}},\vartheta_{\mbf{n}}  }  (t)=\left\langle \widehat{\mf{u}}^{\gamma_{\mbf{n}},\vartheta_{\mbf{n}}  }, \xi_{\mbf{j} } \right\rangle$. 
On the other hand, the last equation of \eqref{prre} gives us
\begin{align*}    
\widehat{\mf{u}}_{\mbf{j}}^{\gamma_{\mbf{n}},\vartheta_{\mbf{n}}  }  (\mc{T}) + \vartheta_{\mbf{n}}\widehat{\mf{u}}_{\mbf{j}}^{\gamma_{\mbf{n}},\vartheta_{\mbf{n}}  }  (0) =\widehat{\varphi}^{\gamma_{\mbf{n}}}_{\mbf{j}},
\end{align*} 
where $\widehat{\varphi}^{\gamma_{\mbf{n}}}_{\mbf{j}} := \left\langle \widehat{\varphi}^{\gamma_{\mbf{n}}}, \xi_{\mbf{j}} \right\rangle $.  
From the two latter equations, one can see that
\begin{align*}    
\widehat{\mf{u}}_{\mbf{j}}^{\gamma_{\mbf{n}},\vartheta_{\mbf{n}}  }  (t)  =   & \dfrac{E_{\al,1} \big(- \lambda_{\mbf{j}} t^\al \big)}{\vartheta_{\mbf{n}} + E_{\al,1} \big(- \lambda_{\mbf{j}} \mc{T}^\al \big)} \widehat{\varphi}^{\gamma_{\mbf{n}}}_{\mbf{j} }
+ \int\limits_{ 0 }^{ t } E_{\al,1} \big( -\lambda_{\mbf{j}} (t-s)^\al \big) f_{\mbf{j}}(\widehat{\mf{u}}^{\gamma_{\mbf{n}},\vartheta_{\mbf{n}}  })(s) \d s 
\\
&-  \int\limits_{ 0 }^{ \mc{T} } \dfrac{E_{\al,1} \big(- \lambda_{\mbf{j}} t^\al \big) E_{\al,1} \big( -\lambda_{\mbf{j}}(\mc{T}-s)^\al \big)}{\vartheta_{\mbf{n}} + E_{\al,1} \big(- \lambda_{\mbf{j}} \mc{T}^\al \big)} f_{\mbf{j}}(\widehat{\mf{u}}^{\gamma_{\mbf{n}},\vartheta_{\mbf{n}}  })(s) \d s.
\end{align*}  
For convenience, we define 
\begin{align}   
\widehat {\mc{A}}_{\vartheta_{\mbf{n}} }(t) g &:= \sum_{\mbf{j} \in \mb{N}^d}  \Bigg( \dfrac{E_{\al,1} \big(- \lambda_{\mbf{j}} t^\al \big)}{\vartheta_{\mbf{n}} + E_{\al,1} \big(- \lambda_{\mbf{j}} \mc{T}^\al \big)} \left\langle g, \xi_{\mbf{j}} \right\rangle   \Bigg) \xi_{\mbf{j}}, \label{defare}
\end{align}
and $\widehat{\mc{D}}_{\vartheta_{\mbf{n}}}(t,s)g:=\widehat{\mc{A}}_{\vartheta_{\mbf{n}} }(t) \mc{B}(s)g$, for $t,s \in [0,\mc{T}]$. Then,  $\widehat{\mf{u}}^{\gamma_{\mbf{n}},\vartheta_{\mbf{n}}  }$ (called regularized solution) satisfies the equation
\begin{align}  \label{opereqre}    \widehat{\mf{u}}^{\gamma_{\mbf{n}},\vartheta_{\mbf{n}}  } (t,x) = \widehat {\mc{A}}_{\vartheta_{\mbf{n}} }(t) \widehat{\varphi}^{\gamma_{\mbf{n}}} (x) + \int\limits_{ 0 }^{ t } \mc{B}&(t-s) f(s,x,\widehat{\mf{u}}^{\gamma_{\mbf{n}},\vartheta_{\mbf{n}}  }(s,x)) \d s \nn \\
&-  \int\limits_{ 0 }^{ \mc{T} } \widehat {\mc{D}}_{\vartheta_{\mbf{n}} }(t,\mc{T}-s)  f(s,x,\widehat{\mf{u}}^{\gamma_{\mbf{n}},\vartheta_{\mbf{n}}  }(s,x)) \d s.
\end{align}
\subsection{Convergence results} 
We now estimate the error between the regularized solution $\widehat{\mf{u}}^{\gamma_{\mbf{n}},\vartheta_{\mbf{n}}  } $ and the sought solution $u$ in two different cases of space under the following assumptions:
\begin{itemize}
	\item [(H1)] $f$ satisfies the globally Lipschitz property, i.e., there exists a positive
	constant $K$ such that 
	\begin{align*}   
	\left\| f(t,\cdot,u_1(t,\cdot)) - f(t,\cdot,u_2(t,\cdot)) \right\|_{{L}^2(\Omega) }  \le K \left\| u_1(t,\cdot) - u_2(t,\cdot) \right\|_{L^2(\Omega)} , \quad u_1,u_2 \in L^2(\Omega).
	\end{align*}
	\item [(H2)] $K \in \left( 0, \mc{Q}^{-1}_{\al,\mc{T}} \right) $ in which $\mc{Q}^{-1}_{\al,\mc{T}}$ is defined in \eqref{ask}.
\end{itemize}
\noindent \textbf{Part A} \textbf{(Convergence rate in $\mb{X}_{\mc{T}}$}). In this part, we give the error estimate in the space $\mb{X}_{\mc{T}}$:
\begin{theorem} Let $\gamma_{\mbf{n}}$, $\vartheta_{\mbf{n}}$, with $\mbf{n}=(n_1,\dots,n_d) \in \mb{N}^d$, satisfying
	\begin{align*}    
	\lim_{|\mbf{n}| \to \infty} \gamma_{\mbf{n}}=\infty, \quad \lim_{|\mbf{n}| \to \infty} \vartheta_{\mbf{n}}=0, \quad \text{and} \quad \lim_{|\mbf{n}| \to \infty}  \dfrac{\gamma_{\mbf{n}}^{d/2}}{\vartheta_{\mbf{n}}^2} \prod_{i=1}^d n_i^{-4\mu_i} = \lim_{|\mbf{n}| \to \infty} \dfrac{\gamma_{\mbf{n}}^{-\mu_\circ}}{\vartheta_{\mbf{n}}^2} = 0.
	\end{align*}
	Assume that $u(0,\cdot) \in H^1(\Omega)$, $\varphi \in H^{\mu_\circ}(\Omega)$ with $\mu_\circ$ is as in Theorem \ref{maintheomul} and the assumptions (H1),(H2) are satisfied. Then, $\left\| \widetilde {\mf{u}} ^{N_1,N_2} - u \right\|^2_{\mb{X}_\mc{T}}$ is of order 
	\begin{align*}    
	\max \left\{ \dfrac{\gamma_{\mbf{n}}^{d/2}}{\vartheta_{\mbf{n}}^2} \prod_{i=1}^d n_i^{-4\mu_i}, \dfrac{\gamma_{\mbf{n}}^{-\mu_\circ}}{\vartheta_{\mbf{n}}^2}, \vartheta_{\mbf{n}} \right\}.
	\end{align*}  	
\end{theorem}
In order to prove the theorem above, we first give some properties for the operators appearing in the equation \eqref{opereqre}.
\begin{lemma} \label{lmqsre}
	Let $g$ be a function in $L^2(\Omega)$. Then
	\begin{align*}  
	\left\| \widehat {\mc{A}}_{\vartheta_{\mbf{n}} } (t) g \right\|_{L^2(\Omega)}  \le \dfrac{\ms{M}_2}{\vartheta_{\mbf{n}} } \left\| g \right\|_{L^2(\Omega)} , \quad \text{for } 0 \le t \le \mc{T},
	\end{align*}  
and
\begin{align*}    
\left\| \widehat {\mc{D}}_{\vartheta_{\mbf{n}} } (t,s) g \right\|_{L^2(\Omega)} \le  \dfrac{\ms{M}^2_2 \mc{T}^\al}{\ms{M}_1 s^\al}   \left\| g  \right\|_{L^2(\Omega)}, \quad \text{for } 0 \le t \le \mc{T},0 < s \le \mc{T}.
\end{align*}  	
\end{lemma} 
\begin{proof}
	Since
	\begin{align*}    
	\dfrac{E_{\al,1} \big(- \lambda_{\mbf{j}} t^\al \big)}{\vartheta_{\mbf{n}} + E_{\al,1} \big(- \lambda_{\mbf{j}} \mc{T}^\al \big)} \le \dfrac{E_{\al,1} \big(- \lambda_{\mbf{j}} t^\al \big)}{\vartheta_{\mbf{n}}} \le \dfrac{\ms{M}_2}{\vartheta_{\mbf{n}}}, \quad 
	\end{align*} 
	and
\begin{align*}    
 \dfrac{E_{\al,1} (-\lambda_{\mbf{j}} t^\al)}{\vartheta_{\mbf{n}} + E_{\al,1} (-\lambda_{\mbf{j}} \mc{T}^\al)} E_{\al,1} (-\lambda_{\mbf{j}} s^\al) \le \dfrac{E_{\al,1} (-\lambda_{\mbf{j}} t^\al)}{ E_{\al,1} (-\lambda_{\mbf{j}} \mc{T}^\al)} E_{\al,1} (-\lambda_{\mbf{j}} s^\al) \le \dfrac{\ms{M}^2_2 \mc{T}^\al}{\ms{M}_1 s^\al},
\end{align*}  
	we have
	\begin{align*}    
	\left\| \widehat {\mc{A}}_{\vartheta_{\mbf{n}} }(t) g \right\|^2_{{L}^2(\Omega)}  \le \dfrac{\ms{M}_2^2}{\vartheta_{\mbf{n}}^2} \sum_{\mbf{j} \in \mb{N}^d}   \left\langle g, \xi_{\mbf{j}} \right\rangle ^2 = \dfrac{\ms{M}_2^2}{\vartheta_{\mbf{n}}^2} \left\| g \right\|^2_{{L}^2(\Omega)}, \quad \text{and} \quad \left\| \widehat {\mc{D}}_{\vartheta_{\mbf{n}} } (t,s) g\right\|_{L^2(\Omega)}  \le  \dfrac{\ms{M}^2_2 \mc{T}^\al}{\ms{M}_1 s^\al} \left\| g \right\|^2_{{L}^2(\Omega)}.
	\end{align*} 
This completes the proof.
\end{proof}
Under  assumptions (H1) and (H2), one can check that the integral equation \eqref{opereqre} has a unique solution $\widehat{\mf{u}}^{\gamma_{\mbf{n}},\vartheta_{\mbf{n}}  } \in \mb{X}_{\mc{T}}$  using Lemma \ref{lmop},  Lemma \ref{lmqsre} and a similar method as in Lemma \ref{uniquelm}. Now, we are ready to prove Theorem \ref{maintheomul}.
\begin{proof} [Proof of Theorem \ref{maintheomul}] Let 
\begin{align}  \label{vre} v^{\vartheta_{\mbf{n}}  } (t,x) = \widehat {\mc{A}}_{\vartheta_{\mbf{n}} }(t) {\varphi} (x) + \int\limits_{ 0 }^{ t } \mc{B}&(t-s) f(s,x,u(s,x)) \d s -  \int\limits_{ 0 }^{ \mc{T} } \widehat {\mc{D}}_{\vartheta_{\mbf{n}} }(t,\mc{T}-s)  f(s,x,u(s,x)) \d s.
\end{align}
Then, we have
\begin{align}    \label{2terms}
\left\| \widehat{\mf{u}}^{\gamma_{\mbf{n}},\vartheta_{\mbf{n}}  }(t,\cdot)  - u (t,\cdot) \right\|_{L^2(\Omega)} \le \left\| \widehat{\mf{u}}^{\gamma_{\mbf{n}},\vartheta_{\mbf{n}}  }(t,\cdot)  - v^{\vartheta_{\mbf{n}}  } (t,\cdot) \right\|_{L^2(\Omega)} + \left\|  v^{\vartheta_{\mbf{n}}  } (t,\cdot) - u(t,\cdot) \right\|_{L^2(\Omega)}.
\end{align}
\textbf{Step 1} (Estimate the first term of \eqref{2terms}). It follows from \eqref{opereqre} and \eqref{vre} that  
\begin{align*}    
&\left\| \widehat{\mf{u}}^{\gamma_{\mbf{n}},\vartheta_{\mbf{n}}  }(t,\cdot)  - v^{\vartheta_{\mbf{n}}  } (t,\cdot) \right\|_{L^2(\Omega)} \\
&\le \left\| \widehat {\mc{A}}_{\vartheta_{\mbf{n}} }(t) \left( \widehat{\varphi}^{\gamma_{\mbf{n}}} - {\varphi} \right)    \right\|_{{L}^2(\Omega) } +  \int\limits_{ 0 }^{ t } \left\| \mc{B}(t-s) \big( f(s,\cdot,\widehat{\mf{u}}^{\gamma_{\mbf{n}},\vartheta_{\mbf{n}}  } (s,\cdot)) -  f(s,\cdot,u(s,\cdot)) \big)   \right\|_{{L}^2(\Omega) } \d s \\
&+ \int\limits_{ 0 }^{ \mc{T} } \left\|  \widehat {\mc{D}}_{\vartheta_{\mbf{n}} }(t,\mc{T}-s)  \big( f(s,\cdot,\widehat{\mf{u}}^{\gamma_{\mbf{n}},\vartheta_{\mbf{n}}  } (s,\cdot)) -  f(s,\cdot,u(s,\cdot)) \big) \right\|_{L^2(\Omega)}  \d s \\
&=: \Theta_1 + \Theta_2 +\Theta_3 .
\end{align*}  
From Lemma \ref{lmqsre}, one gets
\begin{align*}  
\mb{E} \Theta_1^2 =  \mb{E} 
\left\| \widehat {\mc{A}}_{\vartheta_{\mbf{n}} }(t) \left( \widehat{\varphi}^{\gamma_{\mbf{n}}} - {\varphi} \right)    \right\|^2_{{L}^2(\Omega) } \le \dfrac{\ms{M}^2_2}{\vartheta^2_{\mbf{n}}} \mb{E}  \left\| \widehat{\varphi}^{\gamma_{\mbf{n}}} - {\varphi} \right\|^2_{{L}^2(\Omega) }.
\end{align*}  
For the term $\Theta_2$ and $\Theta_3$, by a similar technique as in \eqref{lastb}, one arrives at
\begin{align*}    
\mb{E} \left( \Theta_2+ \Theta_3 \right)^2 \le 2K^2 \ms{M}_2^2 \mc{T}^2 \left( 1+ \dfrac{\ms{M}_2^2}{\ms{M}_1^2(1-\al)^2}  \right) \sup_{t \in [0,\mc{T}]} \mb{E} \left\| \widehat{\mf{u}}^{\gamma_{\mbf{n}},\vartheta_{\mbf{n}}  }(t,\cdot) - u(t,\cdot)  \right\|_{L^2(\Omega)}^2 .
\end{align*}     
Hence
\begin{align} \label{1stterm}   
\mb{E} \left\| \widehat{\mf{u}}^{\gamma_{\mbf{n}},\vartheta_{\mbf{n}}  }(t,\cdot)  - v^{\vartheta_{\mbf{n}}  } (t,\cdot) \right\|_{L^2(\Omega)}^2 &\le 2 \mb{E} \Theta^2_2 + 2 \mb{E} \left( \Theta_2+ \Theta_3 \right)^2 \nn \\
&\le 2 \dfrac{\ms{M}^2_2}{\vartheta^2_{\mbf{n}}} \mb{E}  \left\| \widehat{\varphi}^{\gamma_{\mbf{n}}} - {\varphi} \right\|^2_{{L}^2(\Omega) } + \dfrac{\mc{Q}^2_{\alpha,\mc{T}} }{2} \left\| \widehat{\mf{u}}^{\gamma_{\mbf{n}},\vartheta_{\mbf{n}}  } - u  \right\|_{\mb{X}_{\mc{T}}}^2 .
\end{align}  
\textbf{Step 2} (Estimate the second term of \eqref{2terms}). It follows from \eqref{opereq} and \eqref{vre}  that
\begin{align*}  
&\left\| v^{\vartheta_{\mbf{n}}}(t,\cdot) - u(t,\cdot) \right\|^2_{{L}^2(\Omega)}  \\
&=\Bigg\| \widehat {\mc{A}}_{\vartheta_{\mbf{n}} }(t) {\varphi}  - {\mc{A}}(t) {\varphi} -  \int\limits_{ 0 }^{ \mc{T} } \left(  \widehat {\mc{D}}_{\vartheta_{\mbf{n}} }(t,\mc{T}-s)  f(s,\cdot,u(s,\cdot)) - {\mc{D}}(t,\mc{T}-s)  f(s,\cdot,u(s,\cdot)) \right)  \d s \Bigg\|^2_{{L}^2(\Omega)} \\
&= \Bigg\|  \sum_{\mbf{j} \in \mb{N}^d} \dfrac{\vartheta_{\mbf{n}} E_{\al,1} \big(- \lambda_{\mbf{j}} t^\al \big)}{  \vartheta_{\mbf{n}} + E_{\al,1} \big(- \lambda_{\mbf{j}} \mc{T}^\al \big)  } \dfrac{ \varphi_{\mbf{j}} - \int\limits_{ 0 }^{ \mc{T} }  E_{\al,1} \big( -\lambda_{\mbf{j}}(\mc{T}-s)^\al \big) f_{\mbf{j}}(u)(s)}{E_{\al,1} \big(- \lambda_{\mbf{j}} \mc{T}^\al \big)}\xi_{\mbf{j}} \Bigg\|_{{L}^2(\Omega)}^2.
\end{align*}  
In addition
$  
u_{\mbf{j}}(0) = \dfrac{ \varphi_{\mbf{j}} - \int\limits_{ 0 }^{ \mc{T} }  E_{\al,1} \big( -\lambda_{\mbf{j}}(\mc{T}-s)^\al \big) f_{\mbf{j}}(u)(s)}{E_{\al,1} \big(- \lambda_{\mbf{j}} \mc{T}^\al \big)},
$
and
\begin{align*}    
\dfrac{\vartheta^2_{\mbf{n}}  E^2_{\al,1} \big(- \lambda_{\mbf{j}} t^\al \big)}{ \left( \vartheta_{\mbf{n}} + E_{\al,1} \big(- \lambda_{\mbf{j}} \mc{T}^\al \big)   \right)^2 } \le \dfrac{\vartheta_{\mbf{n}} \ms{M}_2^2}{4  E_{\al,1} \big(- \lambda_{\mbf{j}} \mc{T}^\al \big)} \le \dfrac{\vartheta_{\mbf{n}} \ms{M}_2^2}{4 \ms{M}_1} \left( 1+ \mc{T}^\alpha \right) \lambda_{\mbf{j} }. 
\end{align*}  
Hence
\begin{align}   \label{2ndterm} 
\left\| v^{\vartheta_{\mbf{n}}}(t,\cdot) - u(t,\cdot) \right\|^2_{{L}^2(\Omega)}  =  \sum_{\mbf{j} \in \mb{N}^d} \dfrac{\vartheta^2_{\mbf{n}}  E^2_{\al,1} \big(- \lambda_{\mbf{j}} t^\al \big)}{ \left( \vartheta_{\mbf{n}} + E_{\al,1} \big(- \lambda_{\mbf{j}} \mc{T}^\al \big)   \right)^2 } u^2_{\mbf{j}} (0) \le \dfrac{\vartheta_{\mbf{n}} \ms{M}_2^2}{4 \ms{M}_1} \left( 1+ \mc{T}^\alpha \right) \left\| u(0,\cdot) \right\|^2_{{H}^1(\Omega)}.
\end{align} 
Combining \eqref{2terms}, \eqref{1stterm} and \eqref{2ndterm}, we deduce that
\begin{align*}    
&\mb{E} \left\| \widehat{\mf{u}}^{\gamma_{\mbf{n}},\vartheta_{\mbf{n}}  }(t,\cdot)  - u (t,\cdot) \right\|^2_{L^2(\Omega)} \\
&\le 4 \dfrac{\ms{M}^2_2}{\vartheta^2_{\mbf{n}}} \mb{E}  \left\| \widehat{\varphi}^{\gamma_{\mbf{n}}} - {\varphi} \right\|^2_{{L}^2(\Omega) } + \mc{Q}^2_{\alpha,\mc{T}}  \left\| \widehat{\mf{u}}^{\gamma_{\mbf{n}},\vartheta_{\mbf{n}}  } - u  \right\|_{\mb{X}_{\mc{T}}}^2  +  \dfrac{\vartheta_{\mbf{n}} \ms{M}_2^2}{2 \ms{M}_1} \left( 1+ \mc{T}^\alpha \right) \left\| u(0,\cdot) \right\|^2_{{H}^1(\Omega)},
\end{align*}  
which follows that
\begin{align*}    
(1-\mc{Q}^2_{\alpha,\mc{T}}) \left\| \widehat{\mf{u}}^{\gamma_{\mbf{n}},\vartheta_{\mbf{n}}  } - u  \right\|_{\mb{X}_{\mc{T}}}^2 \le 4 \dfrac{\ms{M}^2_2}{\vartheta^2_{\mbf{n}}} \mb{E}  \left\| \widehat{\varphi}^{\gamma_{\mbf{n}}} - {\varphi} \right\|^2_{{L}^2(\Omega) } +  \dfrac{\vartheta_{\mbf{n}} \ms{M}_2^2}{2 \ms{M}_1} \left( 1+ \mc{T}^\alpha \right) \left\| u(0,\cdot) \right\|^2_{{H}^1(\Omega)}.
\end{align*}  
Now, using Part (a) of Theorem \ref{maintheomul}, we conclude that
\begin{align*}    
(1-\mc{Q}^2_{\alpha,\mc{T}}) \left\| \widehat{\mf{u}}^{\gamma_{\mbf{n}},\vartheta_{\mbf{n}}  } - u  \right\|_{\mb{X}_{\mc{T}}}^2  &\le 4 \dfrac{\ms{M}^2_2}{\vartheta^2_{\mbf{n}}} \left[  {\overline{C} (\mu,\varphi)} \gamma_{\mbf{n}}^{d/2} \prod_{i=1}^d n_i^{-4\mu_i} + 4 \gamma_{\mbf{n}}^{-\mu_\circ} \left\| \varphi \right\|^2_{H^{\mu_0}(\Omega)} \right] \\
&+  \dfrac{\vartheta_{\mbf{n}} \ms{M}_2^2}{2 \ms{M}_1} \left( 1+ \mc{T}^\alpha \right) \left\| u(0,\cdot) \right\|^2_{{H}^1(\Omega)},
\end{align*}  
which implies that $\left\| \widehat{\mf{u}}^{\gamma_{\mbf{n}},\vartheta_{\mbf{n}}  } - u  \right\|_{\mb{X}_{\mc{T}}}^2$ is of order 
\begin{align*}    
\max \left\{ \dfrac{\gamma_{\mbf{n}}^{d/2}}{\vartheta_{\mbf{n}}^2} \prod_{i=1}^d n_i^{-4\mu_i}, \dfrac{\gamma_{\mbf{n}}^{-\mu_\circ}}{\vartheta_{\mbf{n}}^2}, \vartheta_{\mbf{n}} \right\}.
\end{align*}  
\end{proof}

\noindent \textbf{Part B} \textbf{(Convergence rate in $\mb{S}_{\sigma,\mc{T}}$}). The following theorem gives the error estimate in the space $\mb{S}_{\sigma,\mc{T}}$:
\begin{theorem} \label{main2}
	Let $\gamma_{\mbf{n}}$, $\vartheta_{\mbf{n}}$, with $\mbf{n}=(n_1,\dots,n_d) \in \mb{N}^d$, satisfying
	\begin{align*}    
	\lim_{|\mbf{n}| \to \infty} \gamma_{\mbf{n}}=\infty, \quad \lim_{|\mbf{n}| \to \infty} \vartheta_{\mbf{n}}=0, \quad \text{and} \quad \lim_{|\mbf{n}| \to \infty}  \dfrac{\gamma_{\mbf{n}}^{\sigma + d/2}}{\vartheta_{\mbf{n}}^2} \prod_{i=1}^d n_i^{-4\mu_i} = \lim_{|\mbf{n}| \to \infty} \dfrac{\gamma_{\mbf{n}}^{-\mu_\circ}}{\vartheta_{\mbf{n}}^2} = 0.
	\end{align*}
	Assume that there exists a positive constant $\sigma >0$ such that  $u(0,\cdot) \in H^{\sigma+1}(\Omega)$, $\varphi \in H^{\mu_\circ+\sigma}(\Omega)$ with $\mu_\circ$ is as in Theorem \ref{maintheomul} and assumptions (H1),(H2) are satisfied. Then, $\left\| \widetilde {\mf{u}} ^{N_1,N_2} - u \right\|^2_{\mb{S}_{\sigma,\mc{T}}}$ is of order 
\begin{align*}    
\max \left\{ \dfrac{\gamma_{\mbf{n}}^{\sigma+d/2}}{\vartheta_{\mbf{n}}^2} \prod_{i=1}^d n_i^{-4\mu_i}, \dfrac{\gamma_{\mbf{n}}^{-\mu_\circ}}{\vartheta_{\mbf{n}}^2}, \vartheta_{\mbf{n}} \right\}.
\end{align*}   	
\end{theorem}
\begin{proof}
In order to prove the theorem above, we first give a similar estimate as in Lemma \ref{lmqsre}:
\begin{align}  \label{est1}
	\left\| \widehat {\mc{A}}_{\vartheta_{\mbf{n}} } (t) g \right\|_{H^\sigma(\Omega)}  \le \dfrac{\ms{M}_2}{\vartheta_{\mbf{n}} } \left\| g \right\|_{H^\sigma(\Omega)} , \quad \left\| {\mc{B}}(t) g \right\|_{H^\sigma(\Omega)}  \le \ms{M}_2 \left\| g \right\|_{H^\sigma(\Omega)}, \quad \text{for } 0 \le t \le \mc{T},
\end{align}  
	and
\begin{align}    \label{est2}
	\left\| \widehat {\mc{D}}_{\vartheta_{\mbf{n}} } (t,s) g \right\|_{H^\sigma(\Omega)} \le  \dfrac{\ms{M}^2_2 \mc{T}^\al}{\ms{M}_1 s^\al}   \left\| g  \right\|_{H^\sigma(\Omega)}, \quad \text{for } 0 \le t \le \mc{T},0 < s \le \mc{T}.
\end{align}  	
Under  assumptions (H1) and (H2), one can check that the integral equation \eqref{opereqre} has a unique solution $\widehat{\mf{u}}^{\gamma_{\mbf{n}},\vartheta_{\mbf{n}}  } \in \mb{S}_{\sigma,\mc{T}}$ using \eqref{est1}-\eqref{est2} and a similar method as in Lemma \ref{uniquelm}. Next, we will prove Theorem \ref{main2} using a similar technique as in Part A. In this way, we arrive at
\begin{align*}    
\mb{E} \left\| \widehat{\mf{u}}^{\gamma_{\mbf{n}},\vartheta_{\mbf{n}}  }(t,\cdot)  - v^{\vartheta_{\mbf{n}}  } (t,\cdot) \right\|_{H^\sigma(\Omega)}^2 \le 2 \dfrac{\ms{M}^2_2}{\vartheta^2_{\mbf{n}}} \mb{E}  \left\| \widehat{\varphi}^{\gamma_{\mbf{n}}} - {\varphi} \right\|^2_{H^\sigma(\Omega) } + \dfrac{\mc{Q}^2_{\alpha,\mc{T}} }{2} \left\| \widehat{\mf{u}}^{\gamma_{\mbf{n}},\vartheta_{\mbf{n}}  } - u  \right\|_{\mb{S}_{\sigma,\mc{T}}}^2,
\end{align*} 
and 
\begin{align*} 
\left\| v^{\vartheta_{\mbf{n}}}(t,\cdot) - u(t,\cdot) \right\|^2_{H^\sigma(\Omega)}  \le \dfrac{\vartheta_{\mbf{n}} \ms{M}_2^2}{4 \ms{M}_1} \left( 1+ \mc{T}^\alpha \right) \left\| u(0,\cdot) \right\|^2_{{H}^{\sigma+1}(\Omega)},
\end{align*} 
which gives us
\begin{align*}    
(1-\mc{Q}^2_{\alpha,\mc{T}}) \left\| \widehat{\mf{u}}^{\gamma_{\mbf{n}},\vartheta_{\mbf{n}}  } - u  \right\|_{\mb{S}_{\sigma,\mc{T}}}^2 \le 4 \dfrac{\ms{M}^2_2}{\vartheta^2_{\mbf{n}}} \mb{E}  \left\| \widehat{\varphi}^{\gamma_{\mbf{n}}} - {\varphi} \right\|^2_{H^\sigma(\Omega) } +  \dfrac{\vartheta_{\mbf{n}} \ms{M}_2^2}{2 \ms{M}_1} \left( 1+ \mc{T}^\alpha \right) \left\| u(0,\cdot) \right\|^2_{{H}^{\sigma+1}(\Omega)}.
\end{align*} 
Finally, using Part (b) of Theorem \ref{maintheomul}, we conclude that
\begin{align*}    
(1-\mc{Q}^2_{\alpha,\mc{T}}) \left\| \widehat{\mf{u}}^{\gamma_{\mbf{n}},\vartheta_{\mbf{n}}  } - u  \right\|_{\mb{S}_{\sigma,\mc{T}}}^2  &\le 4 \dfrac{\ms{M}^2_2}{\vartheta^2_{\mbf{n}}} \left[  {\overline{C} (\mu,\varphi)} \gamma_{\mbf{n}}^{\sigma+d/2} \prod_{i=1}^d n_i^{-4\mu_i} + 4 \gamma_{\mbf{n}}^{-\mu_\circ} \left\| \varphi \right\|^2_{H^{\sigma+\mu_0}(\Omega)} \right] \\
&+  \dfrac{\vartheta_{\mbf{n}} \ms{M}_2^2}{2 \ms{M}_1} \left( 1+ \mc{T}^\alpha \right) \left\| u(0,\cdot) \right\|^2_{{H}^{\sigma+1}(\Omega)},
\end{align*}  
which implies that $\left\| \widehat{\mf{u}}^{\gamma_{\mbf{n}},\vartheta_{\mbf{n}}  } - u  \right\|_{\mb{S}_{\sigma,\mc{T}}}^2$ is of order 
\begin{align*}    
\max \left\{ \dfrac{\gamma_{\mbf{n}}^{\sigma+d/2}}{\vartheta_{\mbf{n}}^2} \prod_{i=1}^d n_i^{-4\mu_i}, \dfrac{\gamma_{\mbf{n}}^{-\mu_\circ}}{\vartheta_{\mbf{n}}^2}, \vartheta_{\mbf{n}} \right\}.
\end{align*} 
\end{proof}  

 {
\begin{remark}
		The truncation method in this paper is similar to the method in \cite{Tuan5,Minh}. 
	The quasi-boundary value method in this section is more effective and useful than the one in \cite{Yang}.  The advantage of this method is that  it allows us to estimate  the norm on the Hilbert scales $H^\sigma(\Omega)$. As is known, estimates on higher Sobolev spaces such as  $H^\sigma(\Omega)$ is not an easy  task. 
\end{remark}
}

{\subsection{A general filter method in the multi-dimensional case}
Now, we introduce one more regularization method, called a general filter method. The main idea is to replace the quantity $\frac{E_{\al,1}(-\lambda_{\mbf{j}} t^\al)}{E_{\al,1}(-\lambda_{\mbf{j}} {\mc T}^\al)}$ by a new one
	$          
	\mc{L}_{\mbf j}(\vartheta_{\mbf n }) \frac{E_{\al,1}(-\lambda_{\mbf{j}} t^\al)}{E_{\al,1}(-\lambda_{\mbf{j}} {\mc T}^\al)},
	$ 
	with $\mc{L}_{\mbf j}(\vartheta_{\mbf n })$  chosen as in Theorem \ref{filter}. In this way, regularized solutions $\widetilde{w}^{\gamma_{\mbf{n}}, \vartheta_{\mbf n }}$ are obtained as follows
	\begin{align}    \widetilde{w}^{\gamma_{\mbf{n}}, \vartheta_{\mbf n }}(t,x)  = \widetilde {\mbf{A}}_{\vartheta_{\mbf{n}} }(t) \widehat{\varphi}^{\gamma_{\mbf{n}}} (x) + \int\limits_{ 0 }^{ t } \mc{B}&(t-s) f(s,x,\widetilde{w}^{\gamma_{\mbf{n}}, \vartheta_{\mbf n }}(s,x)) \d s \nn \\
	&-  \int\limits_{ 0 }^{ \mc{T} } \widetilde {\mbf{D}}_{\vartheta_{\mbf{n}} }(t,\mc{T}-s)  f(s,x,\widetilde{w}^{\gamma_{\mbf{n}}, \vartheta_{\mbf n }}(s,x)) \d s,
	\end{align}
	where $\gamma_{\mbf{n}}, \vartheta_{\mbf n}$ satisfy   $\lim_{|\mbf{n}| \to \infty} \gamma_{\mbf{n}}=\infty,  \lim_{|\mbf{n}| \to \infty} \vartheta_{\mbf{n}}=0$ and 
	\begin{align}   \label{defab}
	\widetilde {\mbf{A}}_{\vartheta_{\mbf{n}} }(t) g := \sum_{\mbf{j} \in \mb{N}^d}  \Bigg( \mc{L}_{\mbf j}(\vartheta_{\mbf n }) \dfrac{E_{\al,1} \big(- \lambda_{\mbf{j}} t^\al \big)}{ E_{\al,1} \big(- \lambda_{\mbf{j}} \mc{T}^\al \big)} \left\langle g, \xi_{\mbf{j}} \right\rangle   \Bigg) \xi_{\mbf{j}}, \qquad \widetilde{\mbf{D}}_{\vartheta_{\mbf{n}}}(t,s)g:=\widetilde{\mbf{A}}_{\vartheta_{\mbf{n}} }(t) \mc{B}(s)g.
	\end{align}
	\begin{theorem} [Error estimate obtained by general filter method] \label{filter}
		Let $\mc{L}_{\mbf j}(\vartheta_{\mbf n })$ satisfies the following conditions
		\begin{align}     \label{second}      
		\mc{L}_{\mbf j}(\vartheta_{\mbf n }) \frac{E_{\al,1}(-\lambda_{\mbf{j}} t^\al)}{E_{\al,1}(-\lambda_{\mbf{j}} {\mc T}^\al)} \le C_{\dagger}(\vartheta_{\mbf n }), \quad  0 \le 1- \mc{L}_{\mbf j}(\vartheta_{\mbf n }) \le C_{\ddagger}(\vartheta_{\mbf n }) \lambda_{\mbf{j}}^q, \quad \text{for some } q>0,
		\end{align} 
		where $ C_{\dagger}(\vartheta_{\mbf n }), C_{\ddagger}(\vartheta_{\mbf n })$ satisfy 
		\begin{align*}    
		\lim_{|\mbf{n}| \to \infty} C_{\ddagger}(\vartheta_{\mbf n }) = \lim_{|\mbf{n}| \to \infty}  C^2_{\dagger}(\vartheta_{\mbf n }) \gamma_{\mbf{n}}^{\sigma+d/2} \prod_{i=1}^d n_i^{-4\mu_i} = \lim_{|\mbf{n}| \to \infty} C^2_{\dagger}(\vartheta_{\mbf n }) \gamma_{\mbf{n}}^{-\mu_\circ}= 0.
		\end{align*}
		Assume that there exists a positive constant $\sigma >0$ such that  $u(0,\cdot) \in H^{\sigma+2q}(\Omega)$, $\varphi \in H^{\mu_\circ+\sigma}(\Omega)$ with $\mu_\circ$ is as in Theorem \ref{maintheomul} and  assumptions (H1),(H2) are satisfied. Then, $\left\| \widehat{\mf{u}}^{\gamma_{\mbf{n}},\vartheta_{\mbf{n}}  } - u  \right\|_{\mb{S}_{\sigma,\mc{T}}}^2$ is of order 
		\begin{align*}    
		\max \left\{ C^2_{\dagger}(\vartheta_{\mbf n }) \gamma_{\mbf{n}}^{\sigma+d/2} \prod_{i=1}^d n_i^{-4\mu_i}, C^2_{\dagger}(\vartheta_{\mbf n }) \gamma_{\mbf{n}}^{-\mu_\circ}, C^2_{\ddagger}(\vartheta_{\mbf n }) \right\}.
		\end{align*}   
	\end{theorem}
	\begin{proof}
		Let 
		\begin{align}  \label{wre} \widetilde{v}^{\vartheta_{\mbf{n}}  } (t,x) = \widetilde {\mbf{A}}_{\vartheta_{\mbf{n}} }(t) {\varphi} (x) + \int\limits_{ 0 }^{ t } \mc{B}&(t-s) f(s,x,u(s,x)) \d s -  \int\limits_{ 0 }^{ \mc{T} } \widetilde {\mbf {D}}_{\vartheta_{\mbf{n}} }(t,\mc{T}-s)  f(s,x,u(s,x)) \d s.
		\end{align}
		From the definition \eqref{defab},
		the following estimates hold for $0 \le t \le \mc{T}$, $0 < s \le \mc T$  
		\begin{align*}  
		\left\| \widetilde {\mbf{A}}_{\vartheta_{\mbf{n}} } (t) g \right\|_{H^\sigma(\Omega)}  \le C_{\dagger}(\vartheta_{\mbf n }) \left\| g \right\|_{H^\sigma(\Omega)} , \quad  \left\| \widetilde {\mbf{D}}_{\vartheta_{\mbf{n}} } (t,s) g \right\|_{H^\sigma(\Omega)} \le  \dfrac{\ms{M}^2_2 \mc{T}^\al}{\ms{M}_1 s^\al}  \left\| g  \right\|_{H^\sigma(\Omega)}, \quad g\in H^\sigma(\Omega).
		\end{align*}
		By similar techniques as  in the proof of Theorem \ref{main2} and noting that $\left[ 1- \mc{L}_{\mbf j}(\vartheta_{\mbf n }) \right] E_{\al,1} (-\lambda_{\mbf{j} } t^\al) \le C_{\ddagger}(\vartheta_{\mbf n }) \ms{M}_2 \lambda_{\mbf{j}}^q$, one can check that
		\begin{align} \label{cb11}
		\mb{E} \left\| \widetilde{w}^{\gamma_{\mbf{n}},\vartheta_{\mbf{n}}  }(t,\cdot)  - \widetilde{v}^{\vartheta_{\mbf{n}}  } (t,\cdot) \right\|_{H^\sigma(\Omega)}^2 
		&\le 2 C^2_{\dagger}(\vartheta_{\mbf n }) \mb{E}  \left\| \widehat{\varphi}^{\gamma_{\mbf{n}}} - {\varphi} \right\|^2_{H^\sigma(\Omega) } + \dfrac{\mc{Q}^2_{\alpha,\mc{T}} }{2} \left\| \widehat{\mf{u}}^{\gamma_{\mbf{n}},\vartheta_{\mbf{n}}  } - u  \right\|_{\mb{S}_{\sigma,\mc{T}}}^2,
		\end{align}  
		and that
		\begin{align}  \label{cb22}
		\left\| \widetilde{v}^{\vartheta_{\mbf{n}}}(t,\cdot) - u(t,\cdot) \right\|^2_{{L}^2(\Omega)}  =  \sum_{\mbf{j} \in \mb{N}^d} \left[ 1 - \mc{L}_{\mbf j}(\vartheta_{\mbf n}) \right]^2 E^2_{\al,1} (-\lambda_{\mbf{j} } t^\al)  u^2_{\mbf{j}} (0)   \le \ms{M}_2^2 C^2_{\ddagger}(\vartheta_{\mbf n })  \left\| u(0,\cdot) \right\|^2_{{H}^{2q}(\Omega)},
		\end{align} 
		where $\mc{Q}_{\alpha,\mc{T}}$ is defined as in \eqref{ask}. From \eqref{cb11}-\eqref{cb22}, one arrives at
		\begin{align*}    
		(1-\mc{Q}^2_{\alpha,\mc{T}}) \left\| \widehat{\mf{u}}^{\gamma_{\mbf{n}},\vartheta_{\mbf{n}}  } - u  \right\|_{\mb{S}_{\sigma,\mc{T}}}^2  &\le 4 C^2_{\dagger}(\vartheta_{\mbf n }) \left[  {\overline{C} (\mu,\varphi)} \gamma_{\mbf{n}}^{\sigma+d/2} \prod_{i=1}^d n_i^{-4\mu_i} + 4 \gamma_{\mbf{n}}^{-\mu_\circ} \left\| \varphi \right\|^2_{H^{\sigma+\mu_0}(\Omega)} \right] \\
		&+ \ms{M}_2^2 C^2_{\ddagger}(\vartheta_{\mbf n }) \left\| u(0,\cdot) \right\|^2_{{H}^{\sigma+2q}(\Omega)},
		\end{align*}
		which implies that $\left\| \widehat{\mf{u}}^{\gamma_{\mbf{n}},\vartheta_{\mbf{n}}  } - u  \right\|_{\mb{S}_{\sigma,\mc{T}}}^2$ is of order 
		\begin{align*}    
		\max \left\{ C^2_{\dagger}(\vartheta_{\mbf n }) \gamma_{\mbf{n}}^{\sigma+d/2} \prod_{i=1}^d n_i^{-4\mu_i}, C^2_{\dagger}(\vartheta_{\mbf n }) \gamma_{\mbf{n}}^{-\mu_\circ}, C^2_{\ddagger}(\vartheta_{\mbf n }) \right\}.
		\end{align*}  
This completes the proof.		 
	\end{proof}
}

\begin{remark}
	Our problem is restricted to a rectangular geometry for which the eigenvalues and eigenfunctions of the Laplacian are readily available. The analysis here comes from the trigonemetric functions (sine, the cosine function) of eigenfunctions. Lemma 3.1 give the  representations of the exact solution which is given by trigonemetric functions. However, 
	if we let an arbitrary domain $\Omega$ with a $C^{2}$-boundary, the analysis in this paper is not applied and such problem is more difficult. This challenge and open  problem may  be addressed in future works. 
\end{remark}


\section{Numerical example}
In this section, we describe the Fourier truncated method  applied to some examples of finding the function $u=u(t,x)$ satisfying the following conditions
\begin{align}  
&\dfrac{\partial}{\partial t} u(t,x) - \dfrac{\partial^{1-\al}}{\partial t}  \Delta  u(t,x) = f (t,x,u(t,x)), \quad (t,x) \in (0,1) \times \Omega_d,\label{vd1}\\
&u(t,x) = 0, \quad (t,x) \in (0,1) \times \partial\Omega_d,\label{vd2}\\
&u(\mc{T},x) = \varphi(x), \quad x \in \Omega_d, \label{vd3}
\end{align} 
where $\al \in (0,1)$, $t \in (0,1)$ is time variable, $x \in \Omega_d = (0, \pi)^d$ and $x=(x_1,x_2,\dots,x_d)$ is $d$-dimensional variable. 
{\color{black}
	
	
	The discrete form of the problem \eqref{vd1}-\eqref{vd3} is as follows: We divide the domain $(0, \mathcal{T}) \times \Omega_d $ into $N_d$ and $N_t$ subintervals of
	equal length $h_d$ and $h_t$, where $h_d = \dfrac{\pi}{N_d}$ and $h_t = \dfrac{1}{N_t}$, respectively, where $N_d$ is chosen satisfies the random model as follow:
	
	The data $\varphi$ is measured at $n_1 \times n_2 \times \dots \times n_d$ grid points $x_{\mbf{k}} = x_{k_1,k_2,\dots,k_d} \in \Omega$, $\mbf{k}=(k_1,k_2,\dots,k_d) \in \mb{N}^d$,  as follows
	\begin{align*}    
	x_{\mbf{k}} = (\mc{X}_{k_1},\mc{X}_{k_2}, \dots, \mc{X}_{k_d}) = \left(  \dfrac{2 k_1 -1}{2 n_1} \pi,  \dfrac{2 k_2 -1}{2 n_2} \pi , \dots, \dfrac{2 k_d -1}{2 n_d} \pi \right),
	\end{align*} 
	where $k_i = 1,2,\dots,n_i$, $i = 1,2,\dots,d$. Furthermore, the value of $\varphi$ at each  point $x_{\mbf{k} }$ is contaminated by the observation $\Phi^{obs}_{\mbf{k} }$ 
	\begin{align*}    
	\varphi(x_{\mbf{k}})=\varphi\left( \mc{X}_{k_1},\mc{X}_{k_2},\dots,\mc{X}_{k_d} \right) ~ \approx ~  \Phi^{obs}_{k_1,k_2,\dots,k_d} = \Phi^{obs}_{\mbf{k}}.
	\end{align*}  
	The relationship between two kinds of data is described by the random model
	\begin{align}    \label{model}    \Phi^{obs}_{\mbf{k}} = \varphi(x_{\mbf{k}}) + \varepsilon_{\mbf{k}} \ms{W}_{\mbf{k}},
	\end{align}   
	where $\ms{W}_{\mbf{k}}=\ms{W}_{k_1,k_2,\dots,k_d}$ are mutually independent random variables, $\ms{W}_{\mbf{k}}\sim \mc{N}(0,1)$ and $\varepsilon_{\mbf{k}}=\varepsilon_{k_1,k_2,\dots,k_d}$ are positive constants bounded by a positive constant $\varepsilon_{\max}$.
	
	The function\, {\ttfamily randn} may be used to generate a random number drawn from the $N(0, 1)$ distribution in Matlab software. In order to simulate a state of randomness, the command \, {\ttfamily 	randn(’state’,n)} is used. As an  example, one could use \, {\tt randn(8)} to generate a fixed set of random numbers then we get a matrix \, {\tt 8} $\times$ {\tt 8} with the average of the elements is zero (see Table \ref{T1}).
	\begin{table}[h!]
		\begin{center}	
			{\tt	\begin{tabular}{|rrrrrrrr|}	
					\hline
					-2.2207  &  0.7366  &  1.0446 &   1.4055  &  0.7041 &   0.9410   & 0.2901   & 2.1140 \\
					-0.2391  &  0.9553  & -0.8073  &  1.5757  & -2.3110 &   0.4566   & 0.1594  & -0.0590 \\
					0.0687  &  1.9295   & 0.2059   &-1.1114   & 1.8256  &  0.3717   &-1.1562   & 1.9949 \\
					-2.0202  & -0.7453  & -0.9646  & -2.1935  &  0.4909 &  -0.2571   & 0.3421  &  0.3080 \\
					-0.3641  & -0.8984  & -1.5254  &  0.1623  & -0.0380 &  -0.0990   & 0.1795  & -0.1571 \\
					-0.0813  & -3.2625  &  0.0904  & -0.7056  &  0.5892 &   1.3230  &  0.4859  &  0.7204 \\
					-1.9797  & -0.0300  & -0.4829  &  0.3841  &  0.6980 &   1.9087  & -1.4602   &-0.3344 \\
					0.7882   & 0.6134   & 1.2883   &-0.4194   & -0.3295 &   0.4929  &  0.2335 &  -0.4638	\\			
					\hline
			\end{tabular}}
		\end{center}
		\caption{An example of the function {\tt randn(8)}}	
		\label{T1}
	\end{table}

	In the following, we discuss two examples to illustrate of our results.
	
	\subsection{Case 1: $d = 1, \alpha = 0.3$}
	In first case, the source function $f$ and the data $\varphi$ are chosen as
	\begin{align}
	f := -\sin(x)\left[\dfrac{10t^{0.3}}{3\Gamma(0.3)} + 1\right],\quad  \varphi = \sin(x),\quad  \Phi^{obs} = \varphi(x) + 1\% \ms{W},
	\label{c1D}
	\end{align}
	
	so that the exact solution of the problem \eqref{vd1}-\eqref{vd3} is given by
	$u(t,x) = t\sin(x)$.
	
	The eigenvalues $\left\{ \lambda_{j_1} \right\} $ and the eigenvectors $\left\{ \xi_{j_1} \right\} $ are given by
	$$\lambda_{j_1}  = {j_1}^2,\; \xi_{j_1} = \sqrt{\dfrac{2}{\pi}} \sin({j_1}x),\; \text{ for } {j_1} = 1,2... .$$
	
	According to \eqref{unnre}, we have the regularized solution as follows
	\begin{align}  \label{unnre}  
	\widetilde {\mf{u}}^{N_1}(t,x) &= \mb{A}^{N_1}(t) \varphi^{N_1}(x) + \int\limits_{ 0 }^{ t } \mb{B}^{N_1}(t-s) f(s,x) \d s 
	-  \int\limits_{ 0 }^{ \mc{T} } \mb{D}^{N_1}(t,\mc{T}-s)  f(s,x) \d s,
	\end{align}
	where
	\begin{align*}    
	\mb{A}^{N_1}(t) g &:= \sum_{ j_1=1 }^{ N_1 }\Bigg( \dfrac{E_{\al,1} \big(- \lambda_{j_1} t^\al \big)}{E_{\al,1} \big(- \lambda_{j_1} \mathcal{T}^\al \big)} \left\langle g, \xi_{j_1} \right\rangle  \Bigg) \xi_{j_1}, \\
	\mb{B}^{N_1}(t) g &:= \sum_{ j_1=1 }^{ N_1 } \Big(  E_{\al,1} \big( -\lambda_{j_1} t^\al \big) \left\langle g, \xi_{j_1} \right\rangle \Big) \xi_{j_1},
	\end{align*} 
	and $\mb{D}^{N_1}(t,s)g:=\mb{A}^{N_1}(t) \mb{B}^{N_1}(s)g$.   
	
	Before presenting the results of this subsection, we present an approximate methods to support the calculation as follows
	
	In numerical analysis, Simpson's rule is a method for numerical integration. Let $\theta \in L^2(0,\pi)$, we have the following approximation
	\begin{align*}
	\displaystyle\int_{0}^{\pi} \theta(z) \mathrm{d}z \approx \Delta z \left( \dfrac{1}{3}\theta(z_1) + \dfrac{2}{3}\sum_{l=1}^{(N_z+1)/2-1} \theta(z_{2l}) + \dfrac{4}{3}\sum_{l=1}^{(N_z+1)/2}\theta(z_{2l-1}) + \dfrac{1}{3}\theta(z_{N_z+1})\right).
	\end{align*}
	Then the errors are esimated by
	\begin{align*}
	\mathrm{Err}^{N_1}_{n_1} (t)= \sqrt{\frac{1}{n_1} \sum_{i=1}^{n_1} \big[\widetilde u^{N_1}(t,x_i) -u(t,x_i) \big]^2},
	\end{align*}
	where we choose $N_1$ equal to greatest natural numbers less than ${\log n_1}$.
	Figure \ref{1D1} (a) and Figure \ref{1D1100} (a) show the exact and regularized solutions of 
	the problem \eqref{vd1}-\eqref{vd3} with conditions \eqref{c1D} at $t=0.3$, $n_1 = 50$ and $n_1 = 100$, respectively. In addition, the error between the exact and regularized solutions is shown in Figure \ref{1D1} (b) and Figure \ref{1D1100} (b). Moreover, we also present the solutions on $(t,x) \in (0,1) \times (0,\pi)$ in Figure \ref{1D2} (for $n_1 =50$) and Figure \ref{1D2100} (for $n_1 =100$).
	\subsection{Case 2: $d = 2, \alpha = 0.5$}
	In second case, the model concerned subjects to the following source function and final data
	\begin{align}
	f = -2 \sin(x_1)\sin(x_2)\left[1+ \dfrac{2t^{1/2}}{\Gamma(0.5)} \right], \quad  \varphi(x_{1,2}) = \sin(x_1)\sin(x_2), \quad \Phi^{obs}_{1,2} = \varphi(x_{1,2}) + 1.5\% \ms{W}_{1,2},
	\label{c2D}
	\end{align}
	In order to obtain the solution $u(t,x) = t\sin(x_1)\sin(x_2)$ of our problem in this case, we employ the conditions is give by Eq. \eqref{c2D}.
	
	The eigenvalues $\left\{ \lambda_{j_1,j_2} \right\} $ and the eigenvectors $\left\{ \xi_{j_1,j_2} \right\} $ are given by
	$$\lambda_{j_1,j_2}  = {j_1}^2+{j_2}^2,\; \xi_{j_1,j_2} = {\dfrac{2}{\pi}} \sin({j_1}x_1)\sin({j_2}x_2),\; \text{ for } {j_1},j_2 = 1,2... .$$
	
	According to \eqref{unnre}, we have the regularized solution as follows
	\begin{align}  \label{unnre}  
	\widetilde {\mf{u}}^{N_1,N_2}(t,x) &= \mb{A}^{N_1,N_2}(t) \varphi^{N_1,N_2}(x) + \int\limits_{ 0 }^{ t } \mb{B}^{N_1,N_2}(t-s) f(s,x) \d s 
	-  \int\limits_{ 0 }^{ \mc{T} } \mb{D}^{N_1,N_2}(t,\mc{T}-s)  f(s,x) \d s,
	\end{align}
	where
	
	\begin{align*}    
	\mb{A}^{N_1,N_2}(t) g &:= \sum_{ j_1=1 }^{ N_1 } \sum_{ j_2=1 }^{ N_2 } \Bigg( \dfrac{E_{\al,1} \big(- \lambda_{j_1,j_2} t^\al \big)}{E_{\al,1} \big(- \lambda_{j_1,j_2} \mathcal{T}^\al \big)} \left\langle g, \xi_{j_1,j_2} \right\rangle  \Bigg) \xi_{j_1,j_2}, \\
	\mb{B}^{N_1,N_2}(t) g &:= \sum_{ j_1=1 }^{ N_1 } \sum_{ j_2=1 }^{ N_2 } \Big(  E_{\al,1} \big( -\lambda_{j_1,j_2} t^\al \big) \left\langle g, \xi_{j_1,j_2} \right\rangle \Big) \xi_{j_1,j_2},
	\end{align*} 
	and $\mb{D}^{N_1,N_2}(t,s)g:=\mb{A}^{N_1,N_2}(t) \mb{B}^{N_1,N_2}(s)g$.  
	
	Then we have the errors are esimated by
	\begin{align*}
\mathrm{Err}(t):=	\mathrm{Err}_{n_1,n_2}^{N_1,N_2} (t)= \sqrt{\frac{1}{n_1 n_2} \sum_{j_1=1}^{n_1} \sum_{j_2=1}^{n_2} \big[\widetilde u^{N_1,N_2}(x_{j_1},x_{j_2},t) -u(x_{j_1},x_{j_2},t) \big]^2},
	\end{align*}
	where we choose $N_1$ and $N_2$ equal to greatest natural numbers less than ${\log n_1}$ and ${\log n_2}$, respectively.
	
{\color{red}	In this case, we show the results about the regularized solution (see Figure \ref{2D1}-b) at $t=0.3, n_1 = n_2 =50$. We can compare the exact (see Figure \ref{2D1}-a) and regularized solutions thanks to the error of these solutions by the contour graph (see Figure \ref{2D1}-c).} 
	%
	\begin{table}[h!]\label{Txyz}
		\begin{tabular}{c|c|c|c|c}
			\hline\hline
			\multirow{2}{*}{$\mathrm{Err}(t)$} & \multicolumn{2}{c|}{1D ($n_1 =50$)}                & \multicolumn{2}{c}{2D ($n_1 = n_2 = 50$)}                \\ \cline{2-5} 
			& Numerical method  & Theoretical method & Numerical method  & Theoretical method \\ \hline
			$\mathrm{Err}(0.3)$                & 0.011025586961961 &             0.010118040520960       & 0.024935435226306 &          0.054647006945596          \\ \hline
			$\mathrm{Err}(0.5)$                & 0.010580529848833 &             0.020919970167952       & 0.029741170367171 &     0.030198892384940               \\ \hline
			$\mathrm{Err}(0.8)$                & 0.009010268605417  &       0.010238297159259            & 0.040756196453124 &         0.031748168670840           \\ \hline\hline
		\end{tabular}
		\caption {The error between the exact and regularized solutions at $t \in \{ 0.3, 0.5, 0.8\}$}
	\end{table}
	In Table 2, we show the comparison of errors between theoretical method and 
	numerical method for the cases 1D and 2D with $t \in \{ 0.3, 0.5, 0.8\}$. 
	From this table, it shows that errors by the numerical method give better results than errors by the theoretical method. 
	From the aforementioned evidence, we can conclude that the method that we propose is acceptable.

\section*{Acknowledgements}
The first author gratefully acknowledge stimulating discussions with Dr Yavar Kian. The authors would like to thank the reviewers and editor for their constructive comments and
valuable suggestions that improve the quality of our paper. This research was supported by Vietnam National Foundation for Science and Technology Development (NAFOSTED) under grant number 101.02-2019.09.

\end{document}